\documentclass[11pt,english]{smfart}
\usepackage[T1]{fontenc}
\usepackage[english,francais]{babel}
\usepackage{graphicx}
\usepackage{hyperref}
\usepackage{amssymb}
\usepackage{mathtools}
\usepackage{tikz}
\usepackage{amsmath}
\usetikzlibrary{matrix}
\usepackage{smfthm}
\usepackage{mathrsfs}
\usepackage{pb-diagram}
\usepackage{yfonts}
\usepackage[inline]{enumitem}
\usepackage{color}
\usepackage{verbatim}
\usepackage{bbm}
\usetikzlibrary{matrix}
\theoremstyle{plain}

\def\R{\mathbb{R}}
\def\Q{\mathbb{Q}}
\def\C{\mathbb{C}}
\def\Z{\mathbb{Z}}
\def\N{\mathbb{N}}
\def\O{\mathcal{O}}
\def\L{\Lambda}

\def\nm{\lVert\cdot\rVert}

\def\deg{\widehat{\mathrm{deg}}}
\def\vol{\widehat{\mathrm{vol}}}

\def\Div{\mathrm{Div}}
\def\mumin{\widehat{\mu}_{\min}}
\def\mumax{\widehat\mu_{\max}}
\def\numinasy{\nu_{\min}^{\mathrm{asy}}}
\def\limnto{\lim\limits_{n\rightarrow +\infty}}

\def\xanomega{X^{\mathrm{an}}_\omega}

\title{The continuity of $\chi$-volume functions over adelic curves}
\author{Wenbin LUO}
\date{March 2020}

\begin{document}

\maketitle
\begin{abstract}
    In the setting of Arakelov geometry over adelic curves, we introduce the $\chi$-volume function and show some general properties. This article is dedicated to talk about the continuity of $\chi$-volume function. By discussing its relationship with volume function, we prove its continuity around adelic $\Q$-ample $\Q$-Cartier divisors and its continuity in the trivially valued case. The study of the variation of arithmetic Okounkov bodies leads us to its continuous extension on arithmetic surfaces.
\end{abstract}
\tableofcontents
\section{Introduction}
The aim of this paper is to discuss some properties of several asymptotic invariants of arithmetic divisors under the setting of Arakelov geometry over \textit{adelic curves}.

As a method towards arithmetic geometry, Arakelov geometry is inspired by the idea that when we study an arithmetic variety over $\mathrm{Spec}\, \Z$, the set of its $\C$-points which forms a complex variety should be taken into consideration. This can be understood as a "compactification" of $\mathrm{Spec}\, \Z$ by a transcendental point. As the name shows, Arakelov geometry was initiated by Arakelov in order to define an intersection theory on arithmetic surfaces\cite{Arakelov}. His incomplete blueprint was completed by Faltings\cite{Falt} and then generalized to higher dimensional cases by Gillet and Soul\'e\cite{Gillet-Soule}.

In this article, in order to proceed with a discussion in a more general setting, instead of number fields, we use the notion of adelic curves introduced by Chen and Moriwaki \cite{adelic}. 
We can notice that the function field of a projective curve can be treated equally as a number field in the sense that they both have the product formula. The notion of adele rings of global fields are used to describe those phenomena. Based on this prototype, in order to consider a Arakelov geometry to include them all, an adelic curve is defined to be a field $K$ together with a set of its absolute values parametrised by a measure space $(\Omega, \mathcal A, \nu)$ where $\mathcal A$ is a $\sigma$-algebra and $\nu$ is a measure of $\Omega$ (for the examples mentioned above, the measures would be taken to be just counting measures). 

In \cite{adelic}, Chen and Moriwaki generalized a lot of results in arithmetic geometry to the case over adelic curves. One of them is the continuity of volume function of adelic $\R$-Cartier divisors. In birational geometry, the volume function of a divisor $D$ on an $n$-dimensional variety $X$ is given by $$\mathrm{vol}(D)=\limsup_{n\rightarrow +\infty}\frac{h^0(X,nD)}{n^d/d!}$$
to demonstrate the magnitude of the growth of $h^0(nD)$ with respect to $n$. In classical arithmetic geometry, $h^0$ is replaced by the number of small sections, and the continuity of the analogous volume function was proved by Moriwaki \cite{Continuity}. In the setting over adelic curves, 
we choose the \textit{positive degree} $\deg_+$ (which is gained by taking the maximal \textit{Arakelov degree} among all its subspaces) of an adelic vector bundle $\overline E$ as a sensible analogue of $h^0$. In the proof of the continuity of the volume function over adelic curves in \cite{adelic},
as a very useful method to calculate the Arakelov degree, the \textit{Harder-Narasimhan filtration} is implemented and hence introduced the arithmetic \textit{Okounkov bodies} which would lead to a generalized Brunn-Minkowski type inequality.

In birational geometry, for a big divisor $D$ on a projective variety $X$, we can construct Okounkov body $\Delta(D)$(which is a convex body in $\R^d$) whose volume is exactly the volume of the divisor up to a constant depending on the dimension of the variety solely. In \cite{BC_Okounkov}, Boucksom and Chen associated filtered (in our case, mainly Harder-Narasimhan filtration) linear series of a big line bundle with a concave function on its Okoukov body. In particular, for each arithmetic big line bundle, we can get its arithmetic volume by calculating the integral of the positive part of the concave function (the use of positive part comes from positive degrees). 

Then it's natural to consider the case if we replace positive degrees by simply Arakelov degrees. This idea inspired us to introduce the \textit{$\chi$-volume function} by which we wish to give an analogy of Euler characteristics. The ultimate goal of this study is to prove the following formula,
\begin{equation}\lim\limits_{\epsilon\rightarrow 0}\vol_\chi(\overline D+\epsilon\overline E)=\vol_\chi(\overline D)\end{equation}
where $\overline D$ and $\overline E$ are adelic $\R$-Cartier divisors on an arithmetic variety. The definition of adelic $\R$-Cartier divisors would be given in subsection 3.4. But we only shows some results under certain conditions in this article. 

One main method of this paper is to investigate its relationship with arithmetic volume by using a Green function to make the concave transform positive everywhere (in some sense, we are studying the integrablity of the concave transforms). This paper provides several ways to show the existence of this Green function for some certain divisors (mainly ample divisors).

One of the main results is the continuity of $\vol_\chi(\cdot)$ of adelic $\Q$-Cartier $\Q$-ample divisor:
\begin{theo}{cf. Theorem \ref{thm_conty_chi_Q_ample}}
Let $\overline D=(D,g)$ and $\overline E=(E,h)$ be elements in $\widehat{\Div}_\Q(X)$ such that $D$ is a $\Q$-ample $\Q$-Cartier divisor. Then we can describe the following continuity for $\vol_\chi(\cdot)$:
$$\lim\limits_{n\rightarrow \pm\infty}\vol_\chi(\overline D+\frac{1}{n}\overline E)=\vol_\chi(\overline D).$$
\end{theo}

For the case over \textit{trivially valued field}, i.e. the parametrizing measure space $\Omega$ consists of only one point which is the trivial absolute value, the continuity of $\vol_\chi(\cdot)$ is easily gained as shown in subsection 4.2.

We also give the so-called continuous extension for $\chi$-volume on arithmetic surfaces by introducing $\vol_I(\cdot)$ which is nothing but the integral of the concave transform over the Okounkov body. We firstly shows that $\vol_I(\cdot)$ is continuous which would leads us to the following theorem:
\begin{theo}[cf. Corollary \ref{cor_conti_ext_chi}]
Assume that $\dim X=1$. Let $\overline D=(D,g)$ be an adelic $\R$-Cartier divisor on $X$ such that $\mathrm{deg}(D)>0$. So we can write $\overline D$ as $$\overline D=\alpha_1\overline D_1+\alpha_2 \overline D_2+\cdots+\alpha_n \overline D_n$$
where $\overline D_i=(D_i,g_i)$, $D$ is a Cartier divisor on $X$ with $\mathrm{deg}(D_i)>0$, $i=1,\dots, n.$
Then we have the following continuous extension of $\vol_\chi(\cdot)$:
$$\lim\limits_{\substack{a_i\rightarrow \alpha_i\\a_i\in\Q\\ i=1,\dots,n}}\vol_\chi(a_1 \overline D_1+a_2\overline D_2+\cdots+a_n\overline D_n)=\frac{\vol_I(\overline D)}{2}.$$
\end{theo}
\section{Adelic curves and adelic vector bundles}
Let $K$ be a field and $M_K$ be the set of all its absolute values. Let $(\Omega, \mathcal A, \nu)$ be a measure space where $\mathcal A$ is a $\sigma$-algebra on $\Omega$ and $\nu$ is a measure on $(\Omega,\mathcal A)$. If there exists a $\phi: \Omega\rightarrow M_K, \omega\mapsto \lvert\cdot\rvert_\omega$ such that for any $a\in K\backslash\{0\}$, the function $$(\omega\in\Omega)\mapsto \ln|a|_\omega$$ is $\mathcal A$-measurable, integrable with respect to $\nu$, then we call the whole data $(K,(\Omega, \mathcal A, \nu),\phi)$ an \textit{adelic curve}. For each $\omega\in\Omega$, we denote by $K_\omega$ the completion of $K$ with respect to the absolute value $\lvert\cdot\rvert_\omega$.
Moreover, we say $S$ is a \textit{proper adelic curve} if $S$ satisfies the product formula 	$$\int_{\omega\in\Omega} \ln|a|_\omega \nu(d\omega)=0.$$
\begin{exem}
    Let $K$ be a number field and $\Omega=\{\text{all places of }K\}$. We equip $\Omega$ with discrete $\sigma$-algebra $\mathcal{A}$ and $\nu(\{\omega\})=[K_\omega:\Q_\omega]$ where $\omega\in \Omega$ and $K_\omega(\text{resp. }\Q_\omega)$ is the completion of $K(\text{resp. }\Q)$ with respect to the absolute value $\lvert\cdot\rvert_\omega$. Then such an adelic curve $S=(K,(\Omega,\mathcal A,\nu),\phi)$ is a proper adelic curve due to the product formula for number field: 
$$\prod\limits_{\omega\in\Omega}\lvert a\rvert_\omega^{[K_\omega:\Q_\omega]}=1$$
for any $a\in K\backslash\{0\}$.
\end{exem}
\begin{exem}
Let $K$ be any field and $\lvert\cdot\rvert_0$ the trivial valuation on $K$, namely $\lvert x\rvert_0=\begin{cases}1 & x\not=0,\\ 0 & x=0.\end{cases}$ Let $\Omega=\{\lvert\cdot\rvert_0\}$ and $\nu(\{\lvert\cdot\rvert_0\})=1$. This is also a proper adelic curve. 
\end{exem}

\begin{defi}If the absolute value $\lvert\cdot\rvert_\omega$ on $K$ satisfies the inequality $$|x+y|_\omega\leqslant \max\{|x|_\omega,|y|_\omega\}$$ for $x,y\in K$, we say the norm $\lvert\cdot\rvert_\omega$ is \textit{non-Archimedean}. Otherwise, we say $\lvert\cdot\rvert_\omega$ is \textit{Archimedean}. We denote by $\Omega_\infty:=\{\omega\in\Omega\mid \lvert\cdot\rvert_\omega\text{ is Archimedean}\}$ and $\Omega_{\mathrm{fin}}:=\Omega\setminus\Omega_\infty$ the set of \textit{infinite places} and \textit{finite places} respectively .
\end{defi}

Note that for each $\omega\in\Omega_\infty$, $K_\omega=\R$ or $\C$, and there exists a $\kappa(\omega)\in (0,1]$ such that $\lvert\cdot\rvert^{\kappa(\omega)}=\lvert\cdot\rvert_\omega$ where $\lvert\cdot\rvert$ is the usual absolute value on $\R$ or $\C$. Normally, we assume that $\kappa(\omega)=1$, which makes $\nu(\Omega_\infty)<+\infty$ (see Chen and Moriwaki \cite[Proposition 3.1.2]{adelic}). 
\subsection{Adelic vector bundles}
\begin{defi}
Let $E$ be a finitely dimensional vector space over $K$. Then a \textit{norm family} $\xi$ on $E$ is a set of the form $\{\lVert\cdot\rVert_\omega\}_{\omega\in\Omega}$ where each $\lVert\cdot\rVert_\omega$ is a norm on $E_{K_\omega}:=E\otimes_K K_\omega$.
\begin{enumerate}
    \item[(1)] Let $F$ be a vector subspace of $E$, for each $\omega\in\Omega$, we denote by $\lVert\cdot\rVert_{\omega,F}$ the restriction of $\lVert\cdot\rVert_\omega$ on $F_{K_\omega}:=F\otimes_K K_\omega\subset E_{K_\omega}.$ Then we denote by $\xi|_F$ the norm family $\{\lVert\cdot\rVert_{\omega,F}\}_{\omega\in\Omega}$, which is called the restriction of $\xi$ on $F$.
    \item[(2)] Let $G$ be a quotient space of $E$ (i.e. there is a surjective map $\mu: E\twoheadrightarrow G$), for each $\omega$, we denote by $\nm_{\omega,E_{K_\omega}\twoheadrightarrow G_{K_\omega}}$, the quotient norm of $\nm_{\omega}$ on $G_{K_\omega}$ through the induced surjective map $\mu_{\omega}: E_{K_\omega}\twoheadrightarrow G_{K_\omega}$, i.e.
$$\|s\|_{\omega,E_{K_\omega}\twoheadrightarrow G_{K_\omega}}=\inf\limits_{\substack{r\in E_{K_\omega}\\ \mu_{\omega}(r)=s}}\|r\|_\omega$$
for any $s\in G_{K_\omega}$. Then we denote the norm family $\{\nm_{\omega,E_{K_\omega}\twoheadrightarrow G_{K_\omega}}\}_{\omega\in \Omega}$ by $\xi_{E\twoheadrightarrow G}$, which is called the quotient norm family of $\xi$ on $G$.
\end{enumerate}
\end{defi}

\begin{defi}
Let $\overline E=(E,\xi)$ be a pair of a finitely dimensional vector space over $K$ and a norm family on it. Let $\xi^\vee=\{\nm_{\omega,*}\}_{\omega\in\Omega}$ denote the \textit{dual norm family} on $E^\vee$ where $$\|f\|_{\omega,*}=\sup\limits_{s\in E_{K_\omega}\backslash\{0\}} \frac{|f(s)|_\omega}{\|s\|_\omega}$$
for each $\omega\in\Omega$ and $f\in E^\vee_{K_\omega}=(E_{K_\omega})^\vee$. We denote the pair $(E^\vee,\xi^\vee)$ by $\overline E^\vee$.
\end{defi}

For the dual norm family defined above, there is an important property worth noting. Let $\overline E=(E,\xi)$ be a pair of a finitely dimensional vector space over $K$ and a norm family on it. We say $\xi$ is \textit{ultrametric} on $\Omega_{\mathrm{fin}}$ if for any $\omega\in\Omega_{\mathrm{fin}}$, the norm $\nm_\omega$ is ultrametric, i.e. for any $x,y\in E_{K_\omega}$, the following strong triangle inequality holds:
$$\|x+y\|_\omega\leqslant \max\{\|x\|_\omega,\|y\|_\omega\}.$$
It can be verified that the $\xi^\vee$ defined above is always ultrametric on $\Omega_{\mathrm{fin}}$.

\begin{defi}
Let $\overline E_1=(E_1,\xi_1=\{\nm_{1,\omega}\})$ and $\overline E_2=(E_2,\xi_2=\{\nm_{2,\omega}\})$ be two pairs of finitely dimensional vector space over $K$ and norm family on it. Then we define the so-called $\epsilon,\pi$-tensor product of them to be 
$$\overline E_1\otimes_{\epsilon,\pi}\overline E_2:=(E_1\otimes E_2,\xi_1\otimes_{\epsilon,\pi}\xi_2=\{\nm_\omega\}_{\omega\in\Omega})$$
where for each $\omega\in\Omega$ and $s\in (E_1\otimes E_2)_{K_\omega}=(E_1)_{K_\omega}\otimes(E_2)_{K_\omega}$, $$\|s\|_\omega:=\begin{cases}\inf\limits_{s=\sum\limits_i x_i\otimes y_i}\sum\limits_i\|x_i\|_{1,\omega}\|y_i\|_{2,\omega}&\text{   if }\omega\in \Omega_{\infty},\\
\sup\limits_{\substack{(f_1,f_2)\in (E_1)_{K_\omega}^*\times(E_2)_{K_\omega}^*\\ f_1,f_2\not=0}}{\displaystyle\frac{|s(f_1,f_2)|_\omega}{\|f_1\|_{1,\omega,*}\|f_2\|_{1,\omega,*}}}&\text{   if }\omega\in \Omega_{\mathrm{fin}}.
\end{cases}$$
Note that the definition for the second case above is sensible because we can view $s\in(E_1)_{K_\omega}\otimes(E_2)_{K_\omega}$ as a bilinear form on the space $(E_1)_{K_\omega}^*\times(E_2)_{K_\omega}^*$. Furthermore, $\xi_1\otimes_{\epsilon,\pi}\xi_2$ is ultrametric on $\Omega_{\mathrm{fin}}$.
\end{defi}

\begin{defi}
Let $\overline E=(E,\xi)$ be a pair of a vector space over $K$ of dimension $r$ and a norm family on it. Then we define its determinant to be $$\mathrm{det}\overline E:=(\mathrm{det}E,\mathrm{det}\xi=\{\nm_{\omega,\mathrm{det}}\})$$
where for any $\omega\in\Omega$ and $s\in (\mathrm{det}E)_{K_\omega}=\mathrm{det}E_{K_\omega}$, the determinant norm of $s$ is given by
$$\|s\|_{\omega,\mathrm{det}}=\inf\limits_{\substack{s=x_1\wedge\cdots\wedge x_r\\x_1,\cdots,x_r\in E_{K_\omega}}}\|x_1\|_\omega\cdots\|x_r\|_\omega.$$
\end{defi}

\begin{defi}
{\bf(Adelic vector bundles)} Let $\overline E=(E,\xi)$ be a pair of a finitely dimensional vector space over $K$ and a norm family on it. 
We say the norm family $\xi$ is \textit{measurable} if the function $$(\omega\in\Omega)\mapsto\|a\|_\omega$$ is $\mathcal A$-measurable with respect to $\nu$ for any $a\in E\backslash\{0\}$.
We say the norm family $\xi$ is \textit{upper dominated} if
$$\forall s\in E\backslash\{0\}, \lefteqn{\int_\Omega \ln\|s\|_\omega}\lefteqn{\hspace{1.2ex}\rule[ 3.35ex]{1.1ex}{.05ex}}
\phantom{\int_\Omega \ln\|s\|_\omega}\nu(d\omega)<+\infty.$$
Moreover, if both $\xi$ and $\xi^\vee$ are upper dominated, we say $\xi$ is \textit{dominated}.

We say $\overline E$ is an \textit{adelic vector bundle} over $S$ if $\xi$ is both dominated and measurable. Especially when $\dim_K(E)=1$, we call $\overline E$ an \textit{adelic line bundle}.
\noindent Note that we can verify if $\overline E=(E,\xi)$ and $\overline E'$ are adelic vector bundles, then
\begin{itemize}
    \item $(F,\xi|_F)$ and $(G,\xi_{E\twoheadrightarrow G})$ are adelic vector bundles for any vector subspace $F\subset E$ and quotient space $G$ respectively.
    \item The determinant $\mathrm{det}(\overline E)$ is an adelic line bundle.
    \item $\overline E\otimes_{\epsilon,\pi}\overline E'$ is an adelic vector bundle.
\end{itemize}
(see Chen and Moriwaki \cite[Proposition 4.1.32]{adelic}).
\end{defi}

\subsection{Arakelov degree of adelic vector bundles}
Throughout this subsection, $S=(K,(\Omega,\mathcal A,\nu),\phi)$ is a proper adelic curve.

\begin{defi}
Let $(E,\xi)$ be an adelic line bundle over $S$. Then we define its \textit{Arakelov degree} to be $$\widehat{\mathrm{deg}}(E,\xi)=-\int_\Omega \ln\lVert s\rVert_\omega \nu(d\omega)$$ where $s$ is a nonzero element of $E$. The definition is independent with the choice of $s$ because $S$ is proper.

If $(E,\xi)$ is an adelic vector bundle (not necessarily of dimension $1$) over $S$. For any $s\in E\backslash\{0\}$, we define the degree of $s$ as
$$\deg_\xi(s):=-\int_{\Omega}\ln\|s\|_\omega \nu(d\omega)$$
and the Arakelov degree of $(E,\xi)$ as
$$\deg(E,\xi):=\begin{cases}\deg(\mathrm{det}E,\mathrm{det}\xi)&\text{ if }E\not=0,\\\ 0&\text{ if }E=0.\end{cases}$$
\end{defi}

\begin{defi} Let $\overline{E}=(E,\xi)$ be an adelic vector bundle over $S$. For any vector subspace $F$ of $E$, we denote by $\overline{F}=(F,\xi|_F)$ the adelic vector subbundle. Then we define the \textit{positive degree} of $\overline E$ to be 
$$\deg_+(\overline E):=\sup_{F\subset E}\deg(\overline F).$$
Note that since $\deg(0)=0$,  $\deg_+(\overline E)\geqslant 0$.
\end{defi}

We may expect that the Arakelov degree $\deg(\cdot)$ defined above are additive with respect to an exact sequence of adelic vector bundles. But this is not true in general. We discuss this issue by introducing the following notions.

Let $\overline{E}=(E,\xi)$ be an adelic vector bundle over $S$. For each $\omega\in\Omega$, let $$\Delta_\omega(\overline E):=\inf\left\{\frac{\lVert\cdot\rVert'_{\mathrm{det}}}{\lVert\cdot\rVert_{\omega,\mathrm{det}}}\middle\vert \begin{array}{l}\lVert\cdot\rVert'\text{ is norm on }E_{K_\omega}\text{, which is either ultrametric }\\\text{or induced by an inner product}\end{array}\right\},$$
$$
\delta_\omega(\overline E):=\frac{\lVert\cdot\rVert_{\omega,\mathrm{det},*}}{\nm_{\omega,*,\mathrm{det}}}
$$
and
$$\Delta(\overline E):=\int_{\omega\in\Omega} \ln\Delta_\omega(\overline E)\nu(d\omega),$$
$$\delta(\overline E):=\int_{\omega\in\Omega} \ln\delta_\omega(\overline E)\nu(d\omega).$$

We can give the following estimates for $\Delta_\omega(\overline E)$ which are depending only on the rank of $E$.
 (see Chen and Moriwaki \cite[Proposition 1.2.42]{adelic} and \cite[Proposition 1.2.54]{adelic})  $$0\leqslant\ln\Delta_\omega(\overline E)\leqslant\begin{cases} \mathrm{rk}(E)\ln(\mathrm{rk}(E))&\text{ if }\omega\in\Omega_{\mathrm{fin}},\\
\displaystyle\frac{1}{2}\ln\mathrm{rk}(E)&\text{ if }\omega\in\Omega_{\mathrm{\infty}}.
\end{cases}$$
The similar estimates for $\delta_\omega(\overline E)$ are given by
$$0\leqslant\ln\delta_\omega(\overline E)\leqslant\begin{cases}0 &\text{ if }\omega\in\Omega_{\mathrm{fin}},\\\displaystyle\frac{1}{2}\mathrm{rk}(E)\ln(\mathrm{rk}(E))&\text{ if }\omega\in\Omega_{\mathrm{\infty}}.
\end{cases}$$
(see Chen and Moriwaki \cite[Proposition 1.2.46, Proposition 1.2.47 and Remark 1.2.55]{adelic}).

\begin{prop}
Let $(E,\xi)$ be an adelic vector bundle. For a flag of vector subspace of $E$:
$$0=E_0\subset E_1\subset\cdots\subset E_n=E,$$ we denote by $\xi_i$ the restriction norm family of $x$ on $E_i$, and $\eta_i$ the quotient norm family of $\xi_i$ on $E_i/E_{i-1}$, we thus get the following inequality
$$\begin{aligned}\sum\limits_{i=1}^n \deg(E_i/E_{i-1},\eta_i)\leqslant &\deg(E,\xi)\\
\leqslant &\sum\limits_{i=1}^n \left(\deg(E_i/E_{i-1},\eta_i)-\Delta(E_i/E_{i-1},\eta_i)\right)+\Delta(E,\xi).\end{aligned}$$
Moreover, if $\xi$ is ultrametric on $\Omega_{\mathrm{fin}}$, then
$$\begin{aligned}\sum\limits_{i=1}^n \deg(E_i/E_{i-1},\eta_i)\leqslant &\deg(E,\xi)\\
\leqslant &\sum\limits_{i=1}^n \left(\deg(E_i/E_{i-1},\eta_i)-\delta(E_i/E_{i-1},\eta_i)\right)+\delta(E,\xi).\end{aligned}$$

\end{prop}
\begin{proof}
See Chen and Moriwaki \cite[Proposition 4.3.12]{adelic}.
\end{proof}

\begin{defi}
Let $(E,\xi)$ be an adelic vector bundle over $S$. The slope, the maximal slope and the minimal slope of $(E,\xi)$ is defined respectively as follows:
$$\begin{aligned}\widehat{\mu}(E,\xi)&:=\frac{\deg(E,\xi)}{\mathrm{rk}(E)},\\
\mumax(E,\xi)&:=\sup\limits_{0\not=F\subset G}\widehat\mu(F,\xi|_F),\\
\mumin(E,\xi)&:=\inf\limits_{E\twoheadrightarrow G\not=0}\widehat\mu(G,\xi_{E\twoheadrightarrow G}).
\end{aligned}$$

\end{defi}

\begin{rema}\label{rmk_mu_nonneg}
The reason that the above notions are important is that $\mumax$ and $\mumin$ are the starting and ending points of the decreasing real numbers sequence corresponding to Harder-Narasimhan filtration of $(E,\xi)$ about which we are going to talk in the next subsection.

Especially in this paper, we care about $\mumin$ and its asymptotic version which are essential for the relationship between $\deg$ and $\deg_+$ because of the following fact:
$$\text{If }\mumin(E,\xi)\geqslant 0\text{, then }\deg(E,\xi)=\deg_+(E,\xi).$$
Indeed, if there exists $F$ a vector subspace of $E$ such that $\deg(F,\xi|_F)>\deg(E,\xi)$, then consider the following exact sequence
$$0\rightarrow F\rightarrow E\rightarrow G\rightarrow 0$$
where $G=E/F$. 
Since $\deg(F,\xi|_F)+\deg(G,\xi_{E\twoheadrightarrow G})\leqslant \deg(E,\xi)$, one obtains that $\deg(G,\xi_{E\twoheadrightarrow G})<0$ which contradicts the fact that $\mumin(E,\xi)\geqslant 0$.
\end{rema}

\begin{prop}\label{mu_min}
Let $(E,\xi)$ be an adelic vector bundle on $S$ and $\phi$ be an integrable function on $\Omega$. Then  $$\mumin(E,\exp(-\phi)\xi)=\mumin(E,\xi)+\int_\Omega \phi(\omega)d\omega.$$
\end{prop}
\begin{proof}
Firstly, for any adelic vector bundle $(F,\xi_F)$ on $S$, we have $$\widehat{\mu}(F,\exp(-\phi)\xi_F)=\widehat{\mu}(F,\xi_F)+\int_\Omega \phi(\omega)d\omega$$ because $\deg(F,\exp(-\phi)\xi_F)=\deg(F,\xi_F)+\mathrm{dim}_K(F)\int_\Omega \phi(\omega)d\omega$.
Then the statement follows from the fact that $(\phi(\omega)\lVert\cdot\rVert)_{E\twoheadrightarrow G}=\phi(\omega)\lVert\cdot\rVert_{E\twoheadrightarrow G}$ for any non-zero quotient space $G$ of $E$ and $\omega\in\Omega$.
\end{proof}
\subsection{Harder-Narasimhan filtration}
\begin{defi}
Let $\overline{E}=(E,\xi)$ be an adelic vector bundle over $S$. For any vector subspace $F\subset E$, we denote by $\overline F$ the adelic vector bundle $(F,\xi|_F)$. Then we consider the following $\R$-filtration
$$\mathcal{F}_{hn}^t(\overline{E})=\sum\limits_{\substack{0\not=F\subset E,\\
\mumin(\overline{F})\geqslant t}}F$$
which is called the \textit{Harder-Narasimhan filtration} of $\overline E$.
We define the \textit{i-th slope} of $\overline E$ by $$\widehat{\mu}_i(\overline E)=\sup\{t\in \R\mid \dim_K(\mathcal{F}_{hn}^t(\overline{E})\geqslant i\}$$ for $i=1,\cdots,r=\dim_K(E).$
\end{defi}
\begin{rema}
It's easy to check that the function $$(t\in\R)\mapsto \widehat{\mu}_{\min}(\mathcal F_{hn}^t(\overline E),\xi|_{\mathcal F_{hn}^t(\overline E)})$$ is a right continuous piecewise function with maximum $\widehat\mu_{\max}(\overline E)$ and minimum $\widehat\mu_{\min}(\overline E)$.
\end{rema}

\begin{prop}\label{prop_hn_asy}
Let $\overline{E}=(E,\xi)$ be an adelic vector bundle on $S$. Then the inequalities
$$\sum\limits_{i=1}^r\widehat{\mu}_{i}(\overline{E})\leqslant \deg(\overline{E})\leqslant \sum\limits_{i=1}^r\widehat{\mu}_{i}(\overline{E})+\Delta(\overline{E})$$
holds where $r=\dim_K(E)$.
\end{prop}
\begin{proof}
See \cite[Proposition 4.3.49 and Proposition 4.3.50]{adelic}.
\end{proof}

\begin{prop}
Let $\overline E=(E,\xi)$ be an adelic vector bundle on $S$ and $\phi$ an integrable function on $\Omega$.
Then
$$\widehat\mu_i(E,\exp(-\phi)\xi)=\widehat\mu_i(E,\xi)+\int_{\omega\in\Omega}\phi(\omega)d\omega$$
for each $i=1,2,\dots,\dim_K(E)$.
\end{prop}
\begin{proof}
For any nonzero vector subspace $F$ of $E$, according to Proposition \ref{mu_min}, it holds that $$\mumin(F,\exp(-\phi)\xi|_F)=\mumin(F,\xi|_F)+\int_{\omega\in\Omega}\phi(\omega)d\omega.$$
Then by the construction of the Harder-Narasimhan filtration, the assertion is proved.
\end{proof}

\subsection{Successive minima}
\begin{defi}
Let $\overline{E}=(E,\xi)$ be an adelic vector bundle over $S$ and $r=\dim_K(E)$. We call
$$\nu_i(\overline{E}):=\sup\{t\in\R:\dim_K(\mathrm{Span}(\{s\in E\mid \deg_\xi(s)\geqslant t\}))\geqslant i\},\text{ }i=1,\cdots,r$$
the \textit{$i^{th}$-minimum} of $\overline E$.

Since $\nu_1(\overline E)\geqslant\nu_2(\overline E)\geqslant\cdots\geqslant \nu_r(\overline E)$, let $\nu_{\max}(\overline E)=\nu_1(\overline E)$ and $\nu_{\min}(\overline E)=\nu_r(\overline E)$ denote the \textit{maximal minimum} and \textit{minimal minimum} respectively.
If $E=0$, then by convention, we set $\nu_{\max}(\overline E):=-\infty$ and $\nu_{\min}(\overline E):=+\infty$.
It's easy to see that
$\nu_{\min}(\overline E)\geqslant 0$ if and only if there exists a basis $\{e_1,\cdots,e_r\}$ of $E$ such that $\deg_\xi(e_i)\geqslant 0$. In this case, we call $E$ can be generated by small sections.
\end{defi}

\begin{defi} We say the adelic curve $S$ satisfies the \textit{strong Minkowski property} of level $\geqslant C$ where $C\in \R_{\geqslant 0}$ if for any adelic vector bundle $(E,\xi)$ on $S$ with $\xi$ being ultrametric on $\Omega_{\mathrm{fin}}$, then
$$\nu_{\min}(E,\xi)\geqslant \mumin(E,\xi)-C\ln(\mathrm{rk}(E)).$$
\end{defi}

\begin{prop}
For any non-zero adelic vector bundle $(E,\xi)$ on $S$, it holds that 
$$\nu_i(E,\xi)\leqslant \widehat{\mu}_i(E,\xi)$$ for $i=1,\dots,\mathrm{rk}(E)$.
Moreover if $S$ satisfies the strong Minkowski property of level $\geqslant C$, then
$$\nu_i(E,\xi)\geqslant \widehat\mu_i(E,\xi)-C\ln(\mathrm{rk}(E)).$$
\end{prop}
\begin{proof}
See Chen and Moriwaki \cite[Corollary 4.3.77 and Proposition 4.3.79]{adelic}.
\end{proof}

\section{Adelic $\R$-Cartier divisors}
Throughout the section, let $S=(K,(\Omega,\mathcal A, \nu),\phi)$ be a proper adelic curve and $\pi:X\rightarrow\mathrm{Spec} K$ be a morphism of schemes where $X$ is a geometrically integral projective $K$-scheme. For each $\omega\in\Omega$, let $X^{\mathrm{an}}_\omega$ denote the Berkovich space associated to $X_\omega:=X\times_{\mathrm{Spec}K}\mathrm{Spec}(K_\omega)$. We denote by $j_\omega$ the specification map from $X^{\mathrm{an}}_\omega\rightarrow X_\omega$ for each $\omega\in\Omega$ (for details, see \cite{Berkovich}). Note that each $x\in \xanomega$ represents a absolute value on $\kappa(j_\omega(x))$ extended from $\lvert\cdot\rvert_\omega$. We denote it by $\lvert\cdot\rvert_x$, and the completion of $\kappa(j_\omega(x))$ respect to $\lvert\cdot\rvert_x$ by $\widehat\kappa(x)$. The \textit{Berkovich topology} on $\xanomega$ is actually defined to be the most coarse topology to make every function with form $x\mapsto |f(j_\omega(x))|_x$($f$ is a rational function on $X_\omega$) and $j_\omega$ continuous.

\subsection{Green functions and continuous metrics}
For a fixed $\omega\in \Omega$,
we set $\mathscr A(X^{\mathrm{an}}_\omega):=\{j_\omega^{-1}(U)\mid U\subset X_\omega\text{ is open}\}$ where each $j_\omega^{-1}(U)$ is called a \textit{Zariski open subset} of $X^{\mathrm{an}}_\omega$.
For any open subset $U\subset X^{\mathrm{an}}_\omega$, we denote by $C^0(U)$ the set of all continuous function on $U$.
Let $$C^0_{\mathrm{gen}}(\xanomega):=\{(U,f)\mid U\in \mathscr A(\xanomega)\text{ and }f\in C^0(U)\}/\sim$$ where $\sim$ is an equivalent relationship given by
$$(U_1,f_1)\sim (U_2,f_2)\text{ if }f_1=f_2\text{ on }U_1\cap U_2\neq\emptyset.$$

Now we have finished the preparation to define the Green functions on Cartier divisors.
Let $D$ be a Cartier divisor on $X$, we denote by $D_\omega$ the pull-back of $D$ under $X_\omega\rightarrow X$. 

\begin{defi}
For any element $g_\omega$ of $C^0_{\mathrm{gen}}(X^{\mathrm{an}}_\omega)$, if $f$ is a local equation of $D_\omega$ on an open subset $U\subset X_\omega$, then $g_\omega+\ln|f|_x\in C^0_{\mathrm{gen}}(X^{\mathrm{an}}_\omega)$.
We say $g_\omega$ is a \textit{Green function} on $D_\omega$ if for any $f$ as above, $g_\omega+\ln|f|_x$ has a representative defined on $j_\omega^{-1}(U)$.

Consider a Green function family $g=\{g_\omega\}_{\omega\in\Omega}$ parametrised by $\Omega$, we say it is a $D$\textit{-Green function family} if for each $\omega\in\Omega$, $g_\omega$ is a Green function on $D_\omega$.
\end{defi}

We define Green function like this because we want each Green function uniquely determines a continuous metric on the corresponding line bundle. Before getting into that discussion, we firstly give the definition of continuous metric as follows:

\begin{defi}
Let $F$ be a locally free $\mathscr O_{X_\omega}$-module of finite rank, a metric on $F$ is a collection $\phi:=\{\lvert\cdot\rvert_\phi(x)\}_{x\in \xanomega}$ of norms $\lvert\cdot\rvert_\phi(x)$ on $F\otimes_{\mathscr O_{X_\omega}}\widehat\kappa(x)$ respectively.
Further, we say $\phi$ is continuous if for any section $s\in F(U)$ where $U$ is a open subset of $X_\omega$, the function
$$(x\in j_\omega^{-1}(U))\mapsto \lvert s(x)\rvert_\phi(x)$$ is continuous on $j_\omega^{-1}(U)$.
\end{defi}

\begin{rema}
Let $L=\mathscr O_X(D)$, then a $D$-Green function family $g$ uniquely determines a metric family $\phi=\{\phi_\omega\}_{\omega\in\Omega}$ such that
$\phi_\omega$ is a continuous metric on the line bundle $L_\omega:=\mathscr O_{X_\omega}(D_\omega)$. More precisely,
For each $x\in\xanomega$, we take a local equation $f$ defining $D_\omega$ around $j_\omega(x)$, then for any $v\in L\otimes \widehat\kappa(x)$, we can write $v$ as $\lambda\otimes f(j_\omega(x))$ where $\lambda\in\widehat\kappa(x)$. The norm $\lvert\cdot\rvert_{\phi_\omega}$ is given by
$$\lvert v\rvert_{\phi_\omega}(x):=\exp(-g(x)-\ln|f(j_\omega(x))|_x)\lvert\lambda\rvert_x. $$
This is well-defined because if $f_1$ and $f_2$ are two local equation of $D_\omega$ around $j_\omega(x)$, then $f_1(j_\omega(x))$ and $f_2(j_\omega(x))$ are differed by an element $u$ in $\mathscr O_{X_\omega,j_\omega(x)}^*$. Thus for $v=\lambda \otimes f_1(j_\omega(x))=u\lambda \otimes f_2(j_\omega(x))$, we have
$$|f_1(j_\omega(x))|_x^{-1}\lvert\lambda\rvert_x=|f_2(j_\omega(x))|_x^{-1}\lvert u\lambda\rvert_x.$$
The metric $\phi_\omega$ is naturally continuous due to the definition of Berkovich topology.
\end{rema}

\subsection{Dominance of metric families}
\begin{defi}[Fubini-Study Metric]
Let $\overline E=(E,\xi)$ be a pair of a vector space over $K$ of rank $n$ and a norm family on it.
We denote by $\mathbb{P}(E)$ the projective space of $E$, and $\mathcal O_E(1)$ the \textit{tautological bundle} of $\mathbb P(E)$. 

For any $\omega\in \Omega$ and any point $x\in \mathbb P(E)^{\mathrm{an}}_\omega$, we are going to assign a norm on $\mathcal O_E(1)\otimes_{\mathcal O_{\mathbb P(E)}}\widehat\kappa(x)$.
For the first step, we assign a norm $\nm_{\overline{E}}(x)$ on $E\otimes_K \widehat\kappa(x)$ by the following rules:

\begin{enumerate}
    \item[(1)] if $\omega\in\Omega_\infty$, then for any $s\in E\otimes_K \widehat\kappa(x)$, we define
    $$\|s\|_{\overline{E}}(x)=\inf\limits_{\substack{s=k_1 f_1+\dots+k_n f_n,\\ k_i\in \widehat\kappa(x), f_i\in E\\ i=1,\dots,n}}\sum\limits_{i=1}^n\|f_i\|_\omega|k_i|_\omega,$$
    \item[(2)] if $\omega\in\Omega_{\mathrm{fin}}$, then for any $s\in E\otimes_K \widehat\kappa(x)$, we define
    $$\|s\|_{\overline{E}}(x)=\sup\limits_{f\in E_{K_\omega}^*,f\neq 0}\frac{|f(s)|_\omega}{\|f\|_{\omega,*}}.$$
\end{enumerate}

We know that $\mathcal O_E(1)$ is globally generated i.e. $E\otimes_K \mathcal O_{\mathbb P(E)}\rightarrow \mathcal O_E(1)$ is surjective. Thus we can get the surjective homomorphism $E\otimes_K \mathcal O_{\mathbb P(E),j_\omega(x)}\rightarrow \mathcal O_E(1)\otimes_{\mathcal O_{\mathbb{P}(E)}}O_{\mathbb P(E),j_\omega(x)}$ which induces the surjective map
$$E\otimes_K \widehat\kappa(x)\rightarrow \mathcal O_E(1)\otimes_{\mathcal O_{\mathbb P(E)}}\widehat\kappa(x).$$

Then we denote by $\lVert\cdot\rVert_{\overline E, \mathrm{FS}}(x)$ the quotient norm of $\nm_{\overline E}(x)$ which is called the \textit{Fubini-Study norm}. For every $\omega \in\Omega$, the norms described above defines a continuous metric on $j^*_\omega(\mathcal O_E(1))$ which is called the \textit{Fubini-Study metric} of $\mathcal O_E(1)$ (see Chen and Moriwaki \cite[Propostion 2.2.12]{adelic}).
\end{defi}

Let $L$ be a very ample line bundle on $X$. 
Let $(E,\xi)$ be a pair of finitely dimensional vector space over $K$ and a norm family on it. Suppose that there exists a surjective homomorphism of sheaves $E\otimes_{K} \mathcal O_X\rightarrow L$ i.e. there exists a surjective map $\beta:E\rightarrow H^0(X,L)$ because $L$ is globally generated. Then we consider the morphism $X\rightarrow \mathbb P(E)$ which is the composition of $X\rightarrow \mathbb P(H^0(X,L))$ and $\mathbb P(H^0(X,L))\rightarrow \mathbb P(E)$. Assume that $X\rightarrow \mathbb P(E)$ is a closed immersion, then we can equip $L$ with a metric family $\phi=\{\phi_\omega\}_{\omega\in\Omega}$ such that each $\phi_\omega$ is a pull-back of the Fubini-study metric of $\mathcal O_E(1)$ under $\xanomega\rightarrow \mathbb P(E)^{\mathrm{an}}_\omega$. We call $\phi$ the quotient metric family induced by $\overline E$ and $\beta$.

\begin{defi}[Distance between metrics] Let $L$ be a line bundle on $X$. For each $\omega\in\Omega$, let $L_\omega$ be the pull-back of $L$ under $X_\omega\rightarrow X$. If $\phi=\{\phi_\omega\}_{\omega\in\Omega}$ and $\psi=\{\psi_\omega\}_{\omega\in\Omega}$ are two continuous metric families on $L$. Then we define the \textit{distance} between $\phi$ and $\psi$ to be
$$d_\omega(\phi,\psi):=\sup\limits_{x\in \xanomega}|\phi_\omega-\psi_\omega|(x).$$
\end{defi}

\begin{defi}
Let $L$ be an very ample line bundle over $X$ with a continuous metric family $\phi=\{\phi_\omega\}_{\omega\in\Omega}$. Then we say $\phi$ is \textit{dominated} if there exists a pair $\overline E=(E,\xi)$ of finite-dimensional vector space $E$ and norm family $\xi$, and a surjective map $\beta: E\rightarrow H^0(X,L)$ inducing a closed immersion $X\rightarrow \mathbb P(E)$ such that the function 
$$\omega\in\Omega\mapsto d_\omega(\phi,\psi)$$
is $\nu$-dominated where $\psi$ is the quotient metric family induced by $\overline E$ and $\beta$.
\end{defi}

\begin{defi}
Let $L$ be a line bundle over $X$ with a continuous metric family $\phi$.
We say $\phi$ is \textit{dominated} if there exists two pairs $\{(L_i,\phi_i)\}_{i=1,2}$ of a very ample line bundle and a dominated metric family such that $L=L_1-L_2$ and $\phi=\phi_1-\phi_2$.
\end{defi}

\begin{prop}\label{prop_dom}
Let $L$ and $L'$ be line bundles over $X$ with continuous metric family $\phi$ and $\phi'$ respectively.
\begin{enumerate}
    \item[(1)] If $\phi$ is dominated, then the dual metric family $-\phi$ on $L^\vee$ is dominated.
    \item[(2)] If both $\phi$ and $\phi'$ are dominated, then the tensor product metric family $\phi+\phi'$ on $L\otimes L'$ is dominated.
\end{enumerate}
\end{prop}
\begin{proof}
See Chen and Moriwaki \cite[Proposition 6.1.12]{adelic}.
\end{proof}

\begin{defi}
Let $(D,g)$ be a pair of a Cartier divisor and a $D$-Green function family. Then we say $g$ is dominated if $\phi_g$ is a dominated metric family of $\mathcal O_X(D)$.
\end{defi}

\begin{theo}
Let $L$ be a line bundle over $X$ with a dominated metric family $\phi=\{\phi_\omega\}_{\omega\in\Omega}$. For each $\omega\in\Omega$, let $\nm_\omega$ be the sup norm corresponding to $\phi_\omega$ on $H^0(X,L)\otimes_{K} K_\omega$. Then the norm family $\xi=\{\nm_\omega\}_{\omega\in\Omega}$ is dominated.
\end{theo}
\begin{proof}
See Chen and Moriwaki \cite[Theorem 6.1.13]{adelic}
\end{proof}

\subsection{Measurability of metric families} Let $X^{\mathrm{an}}$ denote the Berkovich space associated to $X$ equipped with trivial absolute value $\lvert\cdot\rvert$ and $j:X^{\mathrm{an}}\rightarrow X$ the specification map.
We define 
$$X^{\mathrm{an}}_{0}:=\{x\in X\mid j(x)\text{ is closed}\}.$$

We consider each point $x\in X$ such that $\dim(\overline{j_(x)})=1$. Let $F=\kappa(j(x))$, then $F$ is a finitely generated field over $K$ of transcendental degree $1$. Then there exists a positive real number $q$ satisfies the property that
for any absolute value $\lvert\cdot\rvert$ on $F$ over $K$, there exists a closed point $\xi\in \overline{j(x)}$ such that $|s|=\exp(-q\mathrm{ord}_\xi(s))$ for $s\in F$ (see Neukirch \cite[Proposition II.(3.3)]{Neukirch}). Then we call $q$ the \textit{exponent} of $F$ or $\overline{j(x)}$.

Then set
$$X_{1,\Q}^{\mathrm{an}}:=\{x\in X\mid \dim(\overline{j(x)})=1\text{ and the exponent of }\overline{j(x)}\text{ is rational}\}$$
and
$$X_{\leqslant 1,\Q}^{\mathrm{an}}:=X_{0}^{\mathrm{an}}\cup X_{1,\Q}^\mathrm{an}.$$

\begin{defi}
Let $L$ be a line bundle over $X$ with a continuous metric family $\phi$. Then we say $\phi$ is measurable if $\phi=\{\phi_\omega\}_{\omega\in\Omega}$ satisfies the following two conditions:
\begin{enumerate}
    \item[(1)] For any closed point $P$ of $X$, the norm family $\{\lvert\cdot\rvert_{\phi_\omega}(P)\}_{\omega\in\Omega}$ on $L\otimes_{\mathcal O_X} \kappa(P)$ is measurable.
    \item[(2)] For any point $x\in X^{\mathrm{an}}_{\leqslant 1,\Q}$ and any $s\in L\otimes_{\mathcal O_X} \widehat\kappa(x)$, the function 
    $$(\omega\in\Omega_0)\mapsto |s|_{\phi_\omega}(x)$$
    is $\mathcal A_0$-measurable.
\end{enumerate}
\end{defi}

\begin{defi}
Let $(D,g)$ be a pair of a Cartier divisor and a $D$-Green function family. We say $(D,g)$ is measurable if $\phi_g$ is measurable on $\mathcal O_X(D)$.
\end{defi}

\begin{prop}\label{prop_msb}
Let $L$ and $L'$ be line bundles over $X$ with continuous metric family $\phi$ and $\phi'$ respectively.
\begin{enumerate}
    \item[(1)] If $\phi$ is measurable, then the dual metric family $-\phi$ on $L^\vee$ is measurable.
    \item[(2)] If both $\phi$ and $\phi'$ are measurable, then the tensor product metric family $\phi+\phi'$ on $L\otimes L'$ is measurable.
\end{enumerate}
\end{prop}
\begin{proof}
See Chen and Moriwaki \cite[Proposition 6.1.27]{adelic}.
\end{proof}

\subsection{Adelic $\R$-Cartier divisors}
\begin{defi}
Let $(D,g)$ be a pair of a Cartier divisor on $X$ and a $D$-Green function family. We say $(D,g)$ is an \textit{adelic Cartier divisor} on $X$ if the associated metric of $g$ is both dominated and measurable.
\end{defi}

We denote by $\widehat{\mathrm{Div}}(X)$ the set of all adelic Cartier divisor on $X$. Note that $\widehat{\mathrm{Div}}(X)$ is an abelian group by Proposition \ref{prop_dom} and Proposition \ref{prop_msb}. For any $s\in K(X)^\times$, the function $$\phi_\omega:(x\in X^{an}_\omega)\mapsto \ln|s(x)|_\omega$$ is a Green function on $\mathrm{div}(s)_\omega.$ We can show that the $\mathrm{div}(s)$-Green function family $\{\phi_\omega\}_{\omega\in\Omega}$ is both dominated and measurable. Then we denote by $\widehat{\mathrm{div}}(s)$ the adelic Cartier divisor $(\mathrm{div}(s),\{\phi_\omega\}_{\omega\in\Omega})$, which is called a \textit{principal adelic Cartier divisor}. Let $\widehat{\mathrm{PDiv}}(X)$ denote the set of all such $\widehat{\mathrm{div}}(s)$.

\begin{defi}
We denote by $\widehat{\Div}_\R(X)$ the $\R$-vector space $\widehat{\Div}(X)\otimes_\Z \R$ modulo the subspace generated by the elements of the form
$$(0,g_1)\otimes \lambda_1+\cdots+(0,g_n)\otimes\lambda_n-(0,\lambda_1 g_1+\cdots+\lambda_n g_n)$$
where $\lambda_i\in\R$, $i=1,\dots,n$. We call the elements in $\widehat\Div_\R(X)$ the \textit{adelic $\R$-Cartier divisors} on $X$.
Similarly, let $\widehat{\mathrm{PDiv}}_\R(X)$ denote the subspace of $\widehat{\Div}_\R(X)$ generated by $\widehat{\mathrm{PDiv}}(X)$. $\widehat{\Div}_\Q(X)$ and $\widehat{\mathrm{PDiv}}_\Q(X)$ can be defined following the same way.
\end{defi}

For any $\R$-Caritier divisor $D$, we can define the global section space as follows:
$$H_\R^0(D):=\{f\in K(X)^\times\mid \mathrm{div}(f)+D\geqslant_\R 0\}\cup\{0\}.$$
The conditions of Green function family being dominated and measurable will lead us to the following result:
\begin{theo}
Assume that either the $\sigma$-algebra $\mathcal A$ is discrete, or the field $K$ admits a countable subfield which is dense in every $K_\omega$ with respect to $\lvert\cdot\rvert_\omega$ for every $\omega\in\Omega$. For any $(D,g)\in \widehat{\Div}_\R(X)$ and $\omega\in\Omega$, we consider a norm $\lVert\cdot\rVert_{g_\omega}$ on $H^0_\R(D)\otimes_{K}K_\omega$
$$\lVert\phi\rVert_{g_\omega}:=\sup_{x\in X^{\mathrm{an}}_\omega}\{(\exp(-g_\omega)|\phi|_\omega)(x)\}$$
for $\phi\in H^0_\R(D)\otimes_{K}K_\omega$. Let $\xi_g$ denote the norm family $\{\lVert\cdot\rVert_{g_\omega}\}_{\omega\in\Omega}$. Then the pair $(H^0_\R(D),\xi_g)$ is an adelic vector bundle on $S$.
\end{theo}
\begin{proof}
See Chen and Moriwaki \cite[Theorem 6.2.18]{adelic}.
\end{proof}

\section{$\chi$-Volume function}
In this section, let $S=(K,(\Omega,\mathcal{A},\nu),\Phi)$ be a proper adelic curve satisfies tensorial minimal slope of level $C\geqslant 0$ i.e. for any two adelic vector bundles $\overline E$ and $\overline F$ over $S$, the followings inequality of minimal slopes holds:
$$\mumin(\overline E\otimes_{\epsilon,\pi}\overline F)+C\ln(\mathrm{dim}_K(E\otimes F))\geqslant \mumin(\overline E)+\mumin(\overline F).$$
Let $X$ be a normal and geometrically integral projective $K$-scheme of dimension $d$. We also assume that either $\mathcal A$ is discrete or $K$ admits a subfield $K_0$ which is dense in $K_\omega$ for every $\omega\in\Omega$.

\begin{defi}
Let $(D,g)$ be an adelic $\R$-Cartier divisor on $X$. The volume of $(D,g)$ is defined by $$\vol(D,g):=\limsup\limits_{n\rightarrow +\infty}\frac{\deg_+(H^0_\R(nD),\xi_{ng})}{n^{d+1}/(d+1)!}.$$
The $\chi$-volume of $(D,g)$ is defined by $$\vol_\chi(D,g):=\limsup\limits_{n\rightarrow +\infty}\frac{\deg(H^0_\R(nD),\xi_{ng})}{n^{d+1}/(d+1)!}.$$
\end{defi}

About $\vol(\cdot)$, we recall the following results.

\begin{theo} For any $(D,g)\in \widehat{\Div}_\R(X)$, if $D$ is big, the sequence
$$\left\{\frac{\deg_+(H^0_\R(nD),\xi_{ng})}{n^{d+1}/(d+1)!}\right\}_{n\in\N_+}$$ converges to $\vol(D,g)$.
\end{theo}
\begin{proof}
See Chen and Moriwaki \cite[Theorem 6.4.9]{adelic}.
\end{proof}

\begin{theo}[the continuity of volume function]
For any $\overline D,\overline E_1,\dots,\overline E_n\in\widehat\Div_\R(X)$, it holds that 
$$\lim_{\epsilon_1\rightarrow 0,\dots,\epsilon_n\rightarrow 0} \vol(\overline D+\epsilon_1 \overline E_1+\cdots+\epsilon_n \overline E_n)=\vol(\overline D).$$
\end{theo}
\begin{proof}
See Chen and Moriwaki \cite[Theorem 6.4.24]{adelic}.
\end{proof}

\subsection{Several general properties of $\chi$-volume function}
The following Lemma shows that we can make a shift on $\vol_\chi(\cdot)$ by multiplying the Green function with an integrable function on $\Omega$.
\begin{lemm}\label{lem_multi_phi}
Let $(D,g)$ be an adelic $\R$-Cartier divisor on $X$ and $\phi$ be an integrable function on $\Omega$.
Then $$\vol_\chi(D,\phi+g)=\vol_\chi(D,g)+(d+1)\mathrm{vol}(D)\int_\Omega \phi(\omega)\nu(d\omega).$$
\end{lemm}
\begin{proof}
By definition,
we can do the following calculation that $$\begin{aligned}\vol_\chi(D,\phi+g)&=\limsup\limits_{n\rightarrow +\infty}\frac{\deg(H^0_\R(nD),\exp(-n\phi)\xi_{ng})}{n^{d+1}/(d+1)!}\\
&=\limsup\limits_{n\rightarrow +\infty}\frac{\deg(H^0_\R(nD),\xi_{ng})+n\dim_K(H^0_\R(nD))\int_\Omega \phi(\omega)\nu(d\omega)}{n^{d+1}/(d+1)!}\\
&=\vol_\chi(D,g)+(d+1)\mathrm{vol}(D)\int_\Omega \phi(\omega)\nu(d\omega).\end{aligned}$$
\end{proof}

By this shifting property, we can talk about the relationship between $\vol_\chi(\cdot)$ and $\vol(\cdot)$ mentioned above.
For a $(D,g)\in\widehat{\Div}_{\R}(X)$, we introduce the following asymptotic invariants:
$$\begin{aligned}\widehat{\mu}_{\min}^{\sup}(D,g):=\limsup_{n\rightarrow +\infty} \frac{\widehat{\mu}_{\min}(H^0_\R(nD),\xi_{ng})}{n},\\
\mumin^{\inf}(D,g):=\liminf_{n\rightarrow +\infty} \frac{\mumin(H^0_\R(nD),\xi_{ng})}{n}.
\end{aligned}$$
It's easy to see that if $\mumin^{\inf}(D,g)> 0$, then by Remark \ref{rmk_mu_nonneg}, $$\deg(H^0_\R(nD),\xi_{ng})=\deg_+(H^0_\R(nD),\xi_{ng})$$ for every sufficiently large $n$, thus $\vol_\chi(D,g)=\vol(D,g)$. Moreover if $D$ is big, we can make this result even better by the replacing $\mumin^{\inf}(\cdot)$ with $\mumin^{\sup}(\cdot)$.

\begin{prop}\label{prop_vol_vol_chi}
For a $(D,g)\in\widehat{\Div}_{\R}(X)$, if $D$ is big and $\mumin^{\sup}(D,g)>0$, then
$$\vol_\chi(D,g)=\vol(D,g).$$
In general, if $\mumin^{\sup}(D,g)> -\infty$, then there exists an integrable function $\phi$ on $\Omega$ such that
$$\vol_\chi(D,g+\phi)=\vol(D,g+\phi).$$
\end{prop}
\begin{proof}
By definition, it's trivial that $\vol_\chi(D,g)\leqslant\vol(D,g)$.
On the other hand, since $\mumin^{\sup}(D,g)>0$, there exists an increasing sequence $\{n_k\in\N_+\}_{k\in\N_+}$ such that 
$$\mumin(H^0_\R(n_kD),\xi_{n_kg})>0$$
for any $k\in\N_+$. Then by Remark \ref{rmk_mu_nonneg}, we have $$\deg(H^0_\R(n_kD),\xi_{n_kg})=\deg_+(H^0_\R(n_kD),\xi_{n_kg}).$$
Therefore by definition, $$\vol_\chi(D,g)\geqslant \limsup_{k\rightarrow +\infty}\frac{\deg_+(H^0_\R(n_kD),\xi_{n_kg})}{n_k^{d+1}/(d+1)!}.$$
Note that the right hand side is actually a limit and equals to $\vol(D,g)$, so we get the first assertion proved.

If $\mumin^{\sup}(D,g)> -\infty$, then take an integrable function $\phi$ on $\Omega$ such that $\displaystyle\int_\Omega\phi(\omega)d\omega>-\mumin^{\sup}(D,g)$.
By Proposition \ref{mu_min}, it's obvious that $\mumin^{\sup}(D,\phi+g)> 0$. Therefore we obtain the assertion by the first case.
\end{proof}

Now we can see that $\mumin^{\sup}(D,g)$ is an essential asymptotic invariant for the study of $\vol_\chi(D,g)$. Actually, the continuity of $\mumin^{\sup}(\cdot)$ will lead to the continuity of $\vol_\chi(\cdot)$. In the following we are going to show some properties of $\mumin^{\sup}(\cdot)$.

\begin{defi}\label{dfn_sur_multi}
For an $\R$-Cartier divisor $D$ on $X$, we say $D$ satisfies surjectivity of multiplication maps if the canonical map
$$H^0_\R(nD)\otimes H^0_\R(mD)\rightarrow H^0_\R((n+m)D)$$
is surjective for every $n,m\gg 0$.
\end{defi}

\begin{rema}
If $D$ is ample or globally generated, then $D$ satisfies the surjectivity of multiplication maps. For details, see \cite[Example 1.2.22 and Example 2.1.29]{Positivity}
\end{rema}

\begin{lemm}\label{lem_bottom}
For any $\R$-Cartier adelic divisor $(D,g)$ on $X$, if $D$ satisfies the surjectivity of multiplication maps, then the sequence 
$$\left\{\frac{\mumin(H^0_\R(nD),\xi_{ng})}{n}\right\}_{n\in N_{+}}$$
converges to a number in $\R$.
\end{lemm}
\begin{proof}
Set $\mu_n=\mumin(H^0_\R(nD),\xi_{ng})$.
By definition, the canonical map 
$$H^0_\R(n D)\otimes H^0_\R(m D)\rightarrow H^0_\R((n+m)D)$$ is surjective for every $n,m\gg 0.$
Since $S$ satisfies tensorial minimal slope property of level $C\geqslant 0$, we get the following inequality
$$\mu_{n+m}\geqslant \mu_n+\mu_m-C\ln(r_n)-C\ln(r_m)$$
where $r_n=\dim(H^0_\R(nD))$.

By \cite[Proposition 6.3.15]{adelic}, the sequence $\big\{\displaystyle\frac{\mu_n}{n}\big\}$ converges to an element in $\R\cup\{+\infty\}$. But since {$\mumin^{\sup}(D,g)\leqslant \widehat{\mu}_{\max}^{\mathrm{asy}}(D,g)<+\infty$} (\cite[Proposition 6.2.7 and Proposition 6.4.4]{adelic}), we get the assertion proved.
\end{proof}

\begin{prop}\label{prop_estimate_ample}
Let $\overline D=(D,g)$ and $\overline E=(E,h)$ be elements in $\widehat{\Div}_\Q(X)$,
then we can give the following properties:
\begin{enumerate}
    \item[(1)] If $D$ is a $\Q$-ample or semiample Cartier divisor for some $m\in\N_+$, then for any $q\in \N_+$, we have
    $$\mumin^{\sup}(\frac{1}{q}\overline D)\geqslant \frac{1}{q}\mumin^{\sup}(\overline D)>-\infty.$$
    
    \item[(2)] If both $D$ and $E$ are ample Cartier divisors, then $$\mumin^{\sup}(\overline D+\overline E)=\mumin^{\inf}(\overline D+\overline E)\geqslant \mumin^{\sup}(\overline D)+\mumin^{\sup}(\overline E)=\mumin^{\inf}(\overline D)+\mumin^{\inf}(\overline E).$$
    
    \item[(3)] If $D$ is ample Cartier divisor, then
    $$\mumin^{\sup}(n\overline D)=\mumin^{\inf}(n\overline D)=n\mumin^{\sup}(\overline D)=n\mumin^{\inf}(\overline D)$$
    for any $n\in\N_{+}$.
    
    \item[(4)] If $D$ is ample, then 
     $$\mumin^{\sup}(n\overline D+\overline E)\geqslant n\mumin^{\sup}(\overline D)+C$$
     for some constant $C$.
\end{enumerate}
\end{prop}
\begin{proof}

    (1) By the conditions, there exists an integer $m\in \N_+$ such that $mD$ satisfies the surjectivity of multiplication maps, by Lemma \ref{lem_bottom}, the sequence $$\left\{\frac{\mumin(H^0_\R(nmD),\xi_{nmg})}{n}\right\}_{n\in \N_+}$$ converges to some $\mu\in\R$. Thus $\mumin^{\sup}(\overline D)\geqslant \frac{\mu}{m}>-\infty$.
    It's obvious that $$\mumin^{\sup}(\overline D/q)\geqslant \limsup_{n\rightarrow +\infty}\frac{\mumin(H^0_\R(nD),\xi_{ng})}{nq}=\frac{1}{q}\mumin^{\sup}(\overline D).$$
    
    (2) By \cite[Example 1.2.22]{Positivity}, the canonical homomorphism $$H^0_\R(nD)\otimes H^0_\R(nE)\rightarrow H^0_\R(n(D+E))$$ is surjective for every $n\gg0 $. 
    Then we obtain that $$\begin{aligned}\mumin(H^0_\R(n(D+E)),\xi_{n(g+h)})\geqslant
    \mumin(H^0_\R(nD),\xi_{ng})+\mumin(H^0_\R(nE),\xi_{nh})\\-C\ln(\dim_K(H^0_\R(nD))-C\ln(\dim_K(H^0_\R(nE)).\end{aligned}$$
    Taking a quotient over $n$ on both sides, and let $n\rightarrow +\infty$, we obtain the assertion.
    
    (3) This is obvious.
    
    (4) There exists $n_0\in\N_+$ such that $nD+E$ is ample for any $n\geqslant n_0$. Then 
    $$\begin{aligned}\mumin^{\sup}(n\overline L+\overline E)&\geqslant \mumin^{\sup}((n-n_0)\overline D)+\mumin^{\sup}(n_0\overline D+\overline E)\\&=(n-n_0)\mumin^{\sup}(\overline D)+\mumin^{\sup}(n_0\overline D+\overline E)\end{aligned}$$
    for any $n>n_0$.
\end{proof}

\begin{theo}\label{thm_conty_chi_Q_ample}
Let $\overline D=(D,g)$ and $\overline E=(E,h)$ be elements in $\widehat{\Div}_\Q(X)$ such that $D$ is a $\Q$-ample $\Q$-Cartier divisor. Then we can describe the following continuity for $\vol_\chi(\cdot)$:
$$\lim\limits_{n\rightarrow \pm\infty}\vol_\chi(\overline D+\frac{1}{n}\overline E)=\vol_\chi(\overline D).$$
\end{theo}
\begin{proof}
Observe that we can assume $n>0$ because we can just apply the same reasoning to $-\overline E$.

Take a positive integer $m$ such that $mD$ is ample Cartier and $mE$ is Cartier. We firstly give an estimate to $\mumin^{\sup}(\overline D+\displaystyle\frac{\overline E}{n}).$

There exists an $n_0\in N_+$ such that $nmD+mE$ is ample Cartier for any $n\geqslant n_0$.
Set $\mu=\displaystyle\frac{\mumin^{\sup}(m\overline D)}{m}$, $\eta=\displaystyle\frac{\mumin^{\sup}(n_0m\overline D+m\overline E)}{m}$.
Then, by using Proposition \ref{prop_estimate_ample}, $$\begin{aligned}\mumin^{\sup}(\overline D+\frac{1}{n}\overline E)\geqslant &\frac{1}{mn}\mumin^{\sup}(mn(\overline D+\frac{1}{n}\overline E))\geqslant\\
&\frac{1}{n}((n-n_0)\mu+\eta)=\frac{(n-n_0)\mu}{n}+\frac{\eta}{n}\end{aligned}$$
for every $n\geqslant n_0$.

Take an integrable function $\phi$ on $\Omega$ such that $$\int_\Omega\phi(\omega)d\omega>-\inf_{n\in\N_+}\left\{\frac{(n-n_0)\mu}{n}+\frac{\eta}{n}\right\} \geqslant-\mu.$$
Thus we obtain that $\mumin^{\sup}(\overline D+\frac{1}{n}\overline E+(0,\phi))>0$ for $n\gg0$ by Proposition \ref{mu_min}.
Therefore $$\vol_\chi(\overline D+\frac{1}{n}\overline E+(0,\phi))=\vol(\overline D+\frac{1}{n}\overline E+(0,\phi))$$ for every $n\gg0$ due to Lemma \ref{lem_bottom}.
By the continuity of $\vol(\cdot)$, we have 
\begin{equation}\label{eq_limit_n_pos}\begin{aligned}\limnto \vol_\chi(\overline D+\frac{1}{n}\overline E+(0,\phi))=\vol(\overline D+(0,\phi))\\=\vol_\chi(\overline D+(0,\phi))=\vol_\chi(\overline D)+\mathrm{vol}(D)A.\end{aligned}\end{equation}
On the other hand, set $A=(d+1)\displaystyle\int_\Omega\phi(\omega)d\omega$, we can write the left hand side of equation (\ref{eq_limit_n_pos}) as follows:
\begin{equation}\label{eq_limite_n_pos_2}\begin{aligned}
\limnto \vol_\chi(\overline D+\frac{1}{n}\overline E+(0,\phi))&=\limnto \left(\vol_\chi(\overline D+\frac{\overline E}{n})+\mathrm{vol}(D+\frac{E}{n})A\right)\\
&=\limnto\vol_\chi(\overline D+\frac{\overline E}{n})+\mathrm{vol}(D)A.
\end{aligned}
\end{equation}
The first equation follows from Lemma \ref{lem_multi_phi} and the second equation comes from the continuity of $\mathrm{vol}(\cdot)$.
Thus we obtain that $\lim\limits_{n\rightarrow +\infty}\vol_\chi(\overline D+\frac{1}{n}\overline E)=\vol_\chi(\overline D)$ by comparing (\ref{eq_limit_n_pos}) and (\ref{eq_limite_n_pos_2}).
\end{proof}

In general, for an $\R$-divisor $D$, the sheaf given by $$U\mapsto H^0_\R(U,D):=\{f\in K(X)^\times\mid (\mathrm{div}(f)+D)|_U\geqslant_\R 0\}\cup\{0\}$$ is a coherent sheaf on $X$, we denote this sheaf by $\mathscr{O}_X(D)$.
We fix ample line bundle $\overline L_1, L_2$, then there exists $m_0(D,L_1,L_2)\in N_+$ such that 
$$H^0(X, n L_1)\otimes H^0(X,m L_2+\mathscr O_X(D))\rightarrow H^0(X,n L_1+ m L_2+\mathscr O_X(D))$$
is surjective for any $n,m\geqslant m_0(D,L_1,L_2)$.
Moreover, if we equip $L_1$ and $L_2$ with dominated and measurable continuous metric families $\varphi_1,\varphi_2$ respectively, and $D$ with Green function family $g$, then we have the inequality holds:
$$\begin{aligned}
\mumin(&\pi_*(n L_1+m L_2+\mathscr O_X(D),n\varphi_1+m\varphi_2+\varphi_g))\geqslant\\&\mumin(\pi_*(n L_1,n\varphi_1))+\mumin(\pi_*(m L_2+\mathscr O_X(D),m\phi_2+\varphi_g))\\&-C\ln h^0(nL_1)-C\ln h^0(m L_2+\mathscr O_X(D))\end{aligned}$$

\begin{prop}
Assume that $S$ satisfies the strong tensorial minimal slope property of level $\geqslant C$, let $\overline D=(D,g)$ be an ample adelic divisor and $\overline E=(E,h)$ be an abitrary adelic divisor. It holds that
\begin{enumerate}
    \item[(1)] $$\lim_{n\rightarrow\infty}\frac{\deg(n\overline D+\overline E)}{n^{d+1}/(d+1)!}=\vol_\chi(D,g)$$
    \item[(2)] If $H^0(X,E)\not=0$, then
    $$\deg(H^0(nD+E),\xi_{ng+h})-\deg(H^0(nD),\xi_{ng})\geqslant C n^{d}$$
    
\end{enumerate}
\end{prop}
\begin{proof}
(1) Since $D$ is ample, there exists $n_0\in\N_+$ such that 
$$H^0(X,nD)\otimes H^0(X, mD+E)\rightarrow H^0(X,(n+m)D+E)$$
is surjective for every $n,m\geqslant n_0$.
Let $a_n:=\mumin(H^0_\R(nD+E),\xi_{ng+h})$ and $b_n:=\mumin(H^0_\R(nD),\xi_{ng})$. For any $n\in \N_+$, we can write $n=kn_0+l$ where $l=0,1,\dots,n_0-1$, then it holds that
$$a_{kn_0+l}\geqslant k (b_{n_0}-\delta(n_0))+a_l-\delta(l)$$
where $\delta(n):=C\ln\max\{\dim_K H^0(nD),\dim_K H^0(nD+E)\}$. This implies that 
$$\liminf_{n\rightarrow \infty}\frac{a_n}{n}>-\infty.$$
Thus we can take an integrable function $\phi$ on $\Omega$ such that 
$$A:=\int_\Omega\phi(\omega)\nu(d\omega)>-\min(\liminf_{n\rightarrow \infty}\frac{a_n}{n},\mumin^{\inf}(\overline D)),$$ then
$$\begin{aligned}\deg(H^0(nD+E),\xi_{ng+h})&+\dim_K H^0(nD+E)A\\&=\deg_+(H^0(nD+E),\xi_{n(g+\phi)+h})\end{aligned}$$
and $$\vol_\chi(\overline D)+(d+1)\mathrm{vol}(D)A=\vol(\overline D+(0,\phi)),$$
which proves (1).

(2) For any nonzero element $s\in H^0(X, E)$, we consider the following exact sequence
$$0\rightarrow H^0(X,nD)\xrightarrow{\cdot s} H^0(X,nD+E)\rightarrow E_n\rightarrow 0$$ where $E_n$ is the cokernel of $\cdot s$. Let $F_n$ be the image of $H^0_\R(nD)$ through $\cdot s$. We denote by $\overline F_n$ and $\overline E_n$ the adelic vector bundles with restriction and quotient norm families respectively. 
We can easily see that 
$$\begin{aligned}
\deg&(H^0(nD+E),\xi_{ng+h})\geqslant \deg(\overline F_n)+\deg(\overline E_n)\\
&\geqslant \deg(H^0(nD),\xi_{ng})+\dim_K(H^0(nD))\deg_{\xi_{h}}(s)\\ &+\mumin(H^0(nD+E),\xi_{ng+h})(\dim_K(H^0(nD+E))-\dim_K(H^0(nE))).
\end{aligned}$$
Since $\lvert\dim_KH^0(nD+E)-\dim_K H^0(nD)\rvert<C_0n^{d-1}$,
we conclude that
$$\deg(H^0(nD+E),\xi_{ng+h})-\deg(H^0(nD),\xi_{ng})\geqslant C n^{d}$$
\end{proof}

\begin{lemm}\label{mltply_by_fun}
For any $(E,h)$, $(0,f)\in\widehat{\Div}_\R(X)$, it holds that 
$$\lim\limits_{\epsilon\rightarrow 0}\vol_\chi(E,h+\epsilon f)=\vol_\chi(E,h).$$
\end{lemm}
\begin{proof}
Set
$$\begin{aligned}
\phi(\omega)&=\mathrm{max}\left\{\big\lvert\sup\limits_{x\in X_\omega}\{f_\omega(x)\}\big\rvert,\big\lvert\sup\limits_{x\in X_\omega}\{-f_\omega(x)\}\big\rvert\right\},
\end{aligned}$$
which is an integrable function on $\Omega$.
Then $|f_\omega(x)|\leqslant \phi(\omega)$ for any $x\in X_\omega$ and $\omega\in\Omega$.

According to \cite[Proposition 4.3.17]{adelic}, we have
$$\deg(H^0_\R(nE),e^{n|\epsilon|\phi}\xi_{nh})\geqslant\deg(H^0_\R(nE),\xi_{n(h+\epsilon f)})\geqslant \deg(H^0_\R(nE),e^{-n|\epsilon|\phi}\xi_{nh}).
$$
Since
$$\begin{aligned}
\deg(H^0_\R(nE),e^{n|\epsilon|\phi}\xi_{nh})&=\deg(H^0_\R(nE),\xi_{nh})+n|\epsilon|\mathrm{dim}_K(H^0_\R(nE))A,\\
\deg(H^0_\R(nE),e^{-n|\epsilon|\phi}\xi_{nh})&=\deg(H^0_\R(nE),\xi_{nh})-n|\epsilon|\mathrm{dim}_K(H^0_\R(nE))A\end{aligned}$$
where $A=\int_\Omega \phi(\omega) \nu(d\omega)$,
and $\dim_K(H^0_\R(nE))\sim O(n^d)$,
one obtains the assertion by the definition of $\vol_\chi(\cdot)$.
\end{proof}

Our ultimate goal is to extend this discussion to the case where $D$ is big.
As well known that for any ample divisor $A$, there exists a $m\in\N_+$ and an effective divisor $N$ such that $A+N\sim mD$ (see Lazarsfeld \cite[Corollary 2.2.7]{Positivity}). In order to associate ample divisors and big divisors, we give the following discussion. 

\begin{defi}
Let $(D,g)$ be an adelic $\R$-Cartier divisor on $X$. Then we define the asymptotic $\nu_1$ and $\nu_{\min}$ as
$$\begin{aligned}
\nu_{\min}^{\mathrm{asy}}(D,g)&:=\liminf\limits_{n\rightarrow +\infty} \frac{\nu_{\min}(H^0_\R(nD),\xi_{ng})}{n},\\
\nu_1^{\mathrm{asy}}(D,g)&:=\limsup\limits_{n\rightarrow +\infty} \frac{\nu_1(H^0_\R(nD),\xi_{ng})}{n}.
\end{aligned}
$$
Note that there are two properties:
\begin{enumerate}
    \item[(1)] If $\nu_{\min}^{\mathrm{asy}}(D,g)>0$, then $H_\R^0(nD)$ can be generated by sections with positive Arakelov degree for every sufficiently large $n$.
    \item[(2)] If $S$ satisfies strong Minkowski's property of certain level, then $$\begin{aligned}\nu_{\min}^{\mathrm{asy}}(D,g)&=\mumin^{\inf}(D,g),\\
    \nu_1^{\mathrm{asy}}(D,g)&=\mumax^{\mathrm{asy}}(D,g).
    \end{aligned}
    $$
\end{enumerate}
\end{defi}

\begin{prop}
Let $\overline D=(D,g),\overline E=(E,h)\in \widehat{\mathrm{Div}}_\R(X)$. If $H^0_\R(n(D-E))\not=0$ for $n\gg0$ and $\mumin^{\inf}(D)>-\infty$, then
$$\vol_\chi(\overline D)\geqslant \vol_\chi(\overline E)+(d+1)(\mathrm{vol}(E)\nu_1^{\mathrm{asy}}(\overline D-\overline E)+\mumin^{\inf}(D)(\mathrm{vol}(D)-\mathrm{vol}(E))).$$
\end{prop}
\begin{proof}
For any nonzero element $s\in H^0_\R(n(D-E))$, we consider the following exact sequence
$$0\rightarrow H^0_\R(nE)\xrightarrow{\cdot s} H^0_\R(nD)\rightarrow E_n\rightarrow 0$$ where $E_n$ is the cokernel of $\cdot s$. Let $F_n$ be the image of $H^0_\R(nE)$ under $\cdot s$. We denote by $\overline F_n$ and $\overline E_n$ the adelic vector bundles with restriction and quotient norm families respectively. 
We can easily see that 
$$\begin{aligned}
\deg(H^0_\R(nD),\xi_{ng})&\geqslant \deg(\overline F_n)+\deg(\overline E_n)\\
&\geqslant \deg(H^0_\R(nE),\xi_{nh})+\dim_K(H^0_\R(nE))\deg_{\xi_{n(g-h)}}(s)\\ &+\mumin(H^0(nD),\xi_{ng})(\dim_K(H^0_\R(nD))-\dim_K(H^0_\R(nE))).
\end{aligned}$$
Since $s$ is abitrary, we can deduce that 
$$\begin{aligned}
\deg&(H^0_\R(nD),\xi_{ng})\\
&\geqslant \deg(H^0_\R(nE),\xi_{nh})+\dim_K(H^0_\R(nE))\nu_1(H^0_\R(n(D-E)),\xi_{n(g-h)})\\ &+\mumin(H^0(nD),\xi_{ng})(\dim_K(H^0_\R(nD))-\dim_K(H^0_\R(nE))).
\end{aligned}$$
Thus
$$\begin{aligned}
\vol_\chi&(\overline D)=\limsup_{n\rightarrow \infty}\frac{\deg(H^0_\R(nD),\xi_{ng})}{n^{d+1}/(d+1)!}\\
&\geqslant\vol_\chi(\overline E)+(d+1)(\mathrm{vol}(E)\nu_1^{\mathrm{asy}}(n(\overline D-\overline E))+\mumin^{\inf}(D)(\mathrm{vol}(D)-\mathrm{vol}(E))).
\end{aligned}$$
\end{proof}

\begin{coro}\label{ineq_eff}
Let $(D,g), (E,h)\in \widehat{\mathrm{Div}}_\R(X)$. If $(D-E,g-h)\geqslant_\R 0$ and $\mumin^{\inf}(D,g)>0$, then 
$$\vol_\chi(D,g)\geqslant \vol_\chi(E,h).$$
\end{coro}

\subsection{The continuity over trivially valued field}
Then let us consider the case where $K$ is trivially valued, i.e. $\Omega=\{\lvert\cdot\rvert_0\}$ and $\nu(\{\lvert\cdot\rvert_0\})=1$ where $\lvert\cdot\rvert_0$ is the trivial absolute value. In this case, $\widehat\mu_i(\overline E)=\nu_i(\overline E)$ for any adelic line bundle $\overline E$ with rank $r$ and $i=1,\dots, r$.

Let $X^{\mathrm{an}}$ denote the analytification of $X$ with respect to $\lvert\cdot\rvert_0$. For any divisor $D$ on $X$, we can assign a canonical Green function $g^c_D$ (see Ohnishi \cite[Proposition 3.5.1]{Ohnishi}) which makes $\nu_{\min}^{\mathrm{asy}}(D,g^c_D)=\nu_{\max}^{\mathrm{asy}}(D,g^c_D)=0$. Moreover the map $D\mapsto g^c_D$ is $\R$-linear (see Ohnishi \cite[Proposition 3.5.4]{Ohnishi}).

\begin{lemm}
Assume that $K$ is trivially valued. Let $(D,g^c_D+f)$ be an adelic $\R$-Cartier divisor on $X$, then
$$\numinasy(D,g^c_D+f)\geqslant \inf_{x\in X^{\mathrm{an}}}f(x).$$
\end{lemm}
\begin{proof}
Set $\mu=\inf\limits_{x\in X^{\mathrm{an}}}f(x)$.
Then since for any $n\in\N_+$ and $s\in H^0_\R(nD)$,
$$\deg_{n(g^c_D+\mu)}(s)\leqslant \deg_{n(g^c_D+f)}(s),$$
we get the assertion proved.
\end{proof}
\begin{theo}
Assume that $K$ is trivially valued. The following continuity of $\vol_\chi(\cdot)$ holds:
$$\begin{aligned}\lim\limits_{\lvert\epsilon_1\rvert+\cdots+\lvert\epsilon_n\rvert\rightarrow 0}\vol_\chi((D,g^c_D+f)+\epsilon_1(E_1,g^c_{E_1}+h_1)+\cdots+\epsilon_n(E_n,g^c_{E_n}+h_n))\\=\vol_\chi(D,g^c_D+f)\end{aligned}$$
where $(D,g^c_D+f),(E_1,g^c_{E_1}+h_1),\cdots,(E_n,g^c_{E_n}+h_n)$ are adelic $\R$-Cartier divisors on $X$.
\end{theo}
\begin{proof}
We first observe that
$$\begin{aligned}(D,g^c_D&+f)+\epsilon_1(E_1,g^c_{E_1}+h_1)+\cdots+\epsilon_n(E_n,g^c_{E_n}+h_n)=\\
&(D+\epsilon_1 E_1+\cdots+\epsilon_n E_n,g^c_{D+\epsilon_1 E_1+\cdots+\epsilon_n E_n}+f+\epsilon_1 h_1+\cdots+\epsilon_n h_n)
\end{aligned}$$
due to the $\R$-linearity of canonical Green function (see Ohnishi \cite[Proposition 3.5.4]{Ohnishi}).
Set $$\begin{aligned}\mu=&\inf_{x\in X^\mathrm{an}}f(x)\\
\eta_i=&\inf_{x\in X^\mathrm{an}}\min\{h_i(x),-h_i(x)\},&i=1,\dots,n.\end{aligned}$$
Hence according to the lemma above,
$$\begin{aligned}\inf_{x\in X^\mathrm{an}}(f+\epsilon_1 h_1+\cdots+&\epsilon_n h_n)(x)\geqslant&
\\ &\inf_{x\in X^\mathrm{an}}f(x)+\inf_{x\in X^\mathrm{an}}\epsilon_1 h_1(x)&+\cdots+\inf_{x\in X^\mathrm{an}}\epsilon_n h_n(x)\geqslant\\
&&\mu+\lvert\epsilon_1\rvert\eta_1+\cdots+\lvert\epsilon_n\rvert\eta_n.
\end{aligned}$$
Therefore $\numinasy((D,g^c_D+f)+\epsilon_1(E_1,g^c_{E_1}+h_1)+\cdots+\epsilon_n(E_n,g^c_{E_n}+h_n))$ is uniformly bounded from below for $\lvert\epsilon_1\rvert+\cdots+\lvert\epsilon_n\rvert\ll 0$.
Then we obtain the continuity of $\vol_\chi(\cdot)$ due to the continuity of $\vol(\cdot)$ and $\mathrm{vol}(\cdot)$.
\end{proof}

\section{Applications of arithmetic Okounkov bodies}
The arithmetic Okounkov body of an adelic divisor $(D,g)$ is introduced in \cite{adelic} as a concave function on the Okounkov body of $D$, which can be used to calculate the volume of $(D,g)$, we are going to see more applications in this section.

\subsection{Construction of concave transforms}
This subsection generally rephrases the section 3 of chapter 6 in \cite{adelic}.
Let $X$ be a normal, geometrically integral, projective $K$-scheme of dimension $d$ and admits a regular rational point $x$. 
For an adelic $\R$-Cartier divisor $(D,g)$ on $X$, we can view that $H^0_\R(nD)\subset K(X)\subset \mathrm{Frac}(\mathcal{O}_{X,x})\subset \mathrm{Frac}(\widehat{\mathcal{O}}_{X,x})\subset K((T_1,...,T_d))$.

Let $E_n=H^0_\R(nD)$, then we can consider the graded algebra $$E:=\mathop{\oplus}\limits_{n\in\N}E_n\subset K((T_1,...,T_d))[Y].$$ 
We denote by $gr(n,\alpha)$ the vector subspace of $E$ generated by $T^\alpha Y^n$ where $(n,\alpha)\in \N\times\Z^d$, then $E=\mathop\oplus\limits_{(n,\alpha)\in \N\times\Z^d}gr(n,\alpha).$

Set $$\begin{aligned}\Gamma_D:&=\{(n,\alpha)\in \N\times \Z^d\mid0\not=gr(n,\alpha)\},\\
&\Gamma_{D,n}:=\{\alpha \in \Z^d\mid(n,\alpha)\in\Gamma_D\}\end{aligned}$$
and 
$\Delta(D)$ be the closure of
$\{n^{-1}\alpha\mid(n,\alpha)\in \Gamma_D\}$ in $\R^d$, which is called the Okounkov body of $D$. If let $\eta$ be the Lesbegue measure on $\R^d$, then $\eta(\Delta(D))=\mathrm{vol}(D)$.\\

Then we start to give the construction of the concave transforms.
Let $\lvert\cdot\rvert$ denote the trivial absolute value on $K$
For each $n\in \N$, we denote by $\mathcal{F}_n$ the Harder-Narasimhan $\R$-filtration of $(H^0_\R(nD),\xi_{ng})$.
Then we can define a norm $\lVert\cdot\rVert_{\mathcal{F}_n}$ on $E_n$ over $(K,\lvert\cdot\rvert)$ by 
$$\lVert x\rVert_{\mathcal{F}_n}=e^{-\lambda_{\mathcal{F}_n}(x)}$$
for $x\in E_n$ where $\lambda_{\mathcal{F}_n}(x):=\sup\{t\in\R|x\in\mathcal{F}_n^t\}$.
Actually, the graded normed linear series $V_\bullet:=\mathop\oplus\limits_{n\in \N_+}(E_n,\lVert\cdot\rVert_{\mathcal{F}_n})$ contains all the information we need to construct the concave transform.

We equip $\Z^d$ with lexicographic order, then set
$$\mathcal{G}(n)_{\leqslant \alpha}:=\mathop\oplus\limits_{\beta \in \Gamma_n, \beta\leqslant \alpha} gr(n,\beta)$$
and
$$\mathcal{G}(n)_{< \alpha}:=\mathop\oplus\limits_{\beta \in \Gamma_n, \beta< \alpha} gr(n,\beta).$$
Now we can view $gr(n,\alpha)$ as the quotient space of $\mathcal G(n)_{\leqslant \alpha}$ over $\mathcal G(n)_{<\alpha}$. Thus we can give a quotient norm $\lVert\cdot\rVert_{(n,\alpha)}$ of $\lVert\cdot\rVert_{\mathcal F_n}$ on $gr(n,\alpha)$.

Set $g_{(D,g)}(n,\alpha)=-\ln\lVert s\rVert_{(n,\alpha)}$ where $s\in gr(n,\alpha)\backslash\{0\}$. Since $gr(n,\alpha)$ is of dimension $1$ and $K$ is trivially valued, $g(n,\alpha)$ is thus well-defined. We going to show that if we denote by $\delta$ the function that maps $n\in\N_+$ to $C\ln(\#\Gamma_n)$, then $g$ is $\delta$-suppperadditive i.e. for any $(n,\alpha),(m,\alpha)\in \Gamma$, we have $$g_{(D,g)}(n+m,\alpha+\beta)\geqslant g_{(D,g)}(n,\alpha)+g_{(D,g)}(m,\beta)-\delta(n)-\delta(m).$$

Take non-zero elements $e$ and $f$ of $gr(n,\alpha)$ and $gr(m,\beta)$ respectively, then since $X$ is geometrically integral, $ef$ is a non-zero element of $gr(m+n,\alpha+\beta)$. Thus 
$$\begin{aligned}g_{(D,g)}(m+n,\alpha+\beta)&=\lambda_{\mathcal F_{n+m}}(ef)\\
&=\sup\{t\in\R\mid ef\in\mathcal F_{n+m}^t\}\\
&\geqslant \lambda_{\mathcal F_n}(e)+\lambda_{\mathcal F_n}(f)-\delta(n)-\delta(m)
\end{aligned}$$
where the last inequality comes from the fact that $\mathcal F_n^{t_1}\cdot F_m^{t_2}\subset \mathcal F_{n+m}^{t_1+t_2-\delta(n)-\delta(m)}$ for any $t_1,t_2\in \R$ (see Chen and Moriwaki \cite[Proposition 6.3.25]{adelic}).

So far we finished the preparation of the construction, in order to show how $g_{(D,g)}$ is related with the concave transform, we give the following theorem firstly.
\begin{theo}\label{thm_concave_trans}
There exists a concave function $G_{(D,g)}:\Delta(D)\rightarrow \R$ called the concave transform of $g$ such that for any continuous function $f$ on $\R$ with compact support,
the following holds:
$$\limnto \frac{1}{\#\Gamma_n}\sum\limits_{\alpha\in\Gamma_n} f(n^{-1}g_{(D,g)}(n,\alpha))=\frac{1}{\eta(\Delta(D))}\int_{x\in\Delta(D)^\circ}f(G_{(D,g)}(x))dx.$$
\end{theo}
\begin{proof}
See Chen and Moriwaki \cite[Theorem 6.3.16]{adelic}
\end{proof}

Here we roughly give the construction of $G_{(D,g)}$:
We define $$\widetilde{g}_{(D,g)}(u):=\limsup\limits_{n\rightarrow +\infty}\frac{g_{(D,g)}(nu)}{n}$$ for any $u\in\Gamma$. One can show that $\widetilde{g}_{(D,g)}$ satisfies following properties:
\begin{enumerate}
    \item[(1)] $\widetilde{g}_{(D,g)}(u)=\lim\limits_{n\rightarrow +\infty}\displaystyle\frac{g_{(D,g)}(nu)}{n}$,
    \item[(2)] $\widetilde{g}_{(D,g)}(n,\gamma)\geqslant g_{(D,g)}(n,\gamma)-\delta(n)$,
    \item[(3)] $\widetilde{g}_{(D,g)}(u_1+u_2)\geqslant \widetilde{g}_{(D,g)}(u_1)+\widetilde{g}_{(D,g)}(u_2)$.
\end{enumerate}
For any $t\in\R$, set $\Gamma^t_{(D,g)}:=\{(n,\alpha)\in \Gamma|\widetilde{g}_{(D,g)}(n,\alpha)\geqslant nt\}$.
Then we can similarly define the convex body corresponding to $\Gamma^t_{(D,g)}$ as
$$\Delta(\Gamma^t_{(D,g)})=\overline{\{n^{-1}\alpha\mid(n,\alpha)\in \Gamma^t_{(D,g)}\}}.$$

As the family $\{\Delta(\Gamma^t_{(D,g)})\}_{t\in \R}$ of convex bodies is decreasing, the concave transform is given by $G_{(D,g)}(x)=\sup\{t\in\R|x\in\Delta(\Gamma^t_{(D,g)})\}$.

Next we are going to see its relationship with our main goal.
Let $r_n:=\#\Gamma_n=\dim_K(H_\R^0(nD))$ and $\{\alpha_1,\alpha_2,\cdots,\alpha_{r_n}\}$ be the sorted sequence of $\{\alpha\in\Z^d\mid(n,\alpha)\in\Gamma\}$ by lexicographic order. Then we have the flag
\begin{equation}
    0\subsetneq \mathcal G(n)_{\leqslant \alpha_1}\subsetneq\cdots\subsetneq \mathcal G(n)_{\leqslant \alpha_{r_n}}=E_n.
\end{equation}
Since $(K,\lvert\cdot\rvert)$ is spherically complete, by \cite[Proposition 1.2.30]{adelic} , we can find an orthogonal basis ${\bf e}=\{e_1,\dots,e_{r_n}\}$ of $E_n$ compatible with the above flag i.e. $\bf e$ satisfies the following two properties:
\begin{enumerate}
    \item[(1)] $\#{\bf e}\cap \mathcal G(n)_{\leqslant\alpha_i}=i$ for $i=1,\dots,r_n$.
    \item[(2)] For any $x=\lambda_1 e_1+\cdots +\lambda_{r_n}e_{r_n}\in E_n$ where $\lambda_i\in K$ for $i=1,\dots,r_n$, it holds that
    $$\lVert x\rVert_{\mathcal{F}_n}\geqslant \max\limits_{\substack{\lambda_i\neq 0\\ i=1,\dots,r_n}}\{\lVert e_i \rVert\}.$$ 
\end{enumerate}

We can even assume that $e_i\in gr(n,\alpha_i)$ for $i=1,\dots,r_n$.
Then by the property (2) above, we have $\|e_i\|_{(n,\alpha)}=\|e_i\|_{\mathcal{F}_n}$ which implies that the sorted sequence of $\{-\ln\|e_i\|\}_{1\leqslant i\leqslant n}=\{g_{(D,g)}(n,\alpha)\}_{\alpha\in\Gamma_n}$ is identified with $\{\widehat{\mu}_i(E_n,\xi_{ng})\}_{1\leqslant i\leqslant r_n}$. 

Note that by \cite[Proposition 6.2.16, Proposition 6.4.4 and Lemma 6.4.17]{adelic}, we have $$\sup\limits_{x\in\Delta(D)^\circ} G_{(D,g)}(x)=\mumax^{\mathrm{asy}}(D,g):=\limsup_{n\rightarrow +\infty} \frac{\mumax(H^0_\R(nD),\xi_{ng})}{n}<+\infty.$$
Thus we can deduce that 
\begin{theo}\label{thm_converge} For a $(D,g)\in \widehat\Div_\R(X)$ with $D$ being big and $\mumin^{\inf}(D,g)>-\infty$, it holds that
$$\limnto \frac{1}{n^d/d!}\sum\limits_{1\leqslant i\leqslant r_n}\frac{\widehat\mu_i(H^0_\R(nD),\xi_{ng})}{n}=\int_{x\in\Delta(D)^\circ} G_{(D,g)}(x)dx.$$
\end{theo}
\begin{proof}
See Chen and Moriwaki \cite[Remark 6.3.27]{adelic}.
\end{proof}

\begin{coro}\label{cor_converge}
For a $(D,g)\in \Div_\R(X)$, if $D$ is big and $\mumin^{\inf}(D,g)>-\infty$, then
$$\vol_\chi(D,g)=(d+1)\int_{x\in\Delta(D)^\circ}G_{(D,g)}(x)dx.$$
\end{coro}
\begin{proof}
Let $r_n=\dim_K(H^0_\R(nD))$.
By Proposition \ref{prop_hn_asy},
$$\left|\deg(H^0_\R(nD),\xi_{ng})-\sum\limits_{i=1}^{r_n}\widehat\mu_i(H^0_\R(nD),\xi_{ng})\right|\leq \frac{1}{2}r_n\mathrm{ln}(r_n)\nu(\Omega_{\infty}),$$
it follows that $$\begin{aligned}\limsup\limits_{n\rightarrow +\infty}\frac{\deg(H^0_\R(nD),\xi_{ng})}{n^{d+1}/(d+1)!}&=\limsup\limits_{n\rightarrow +\infty}\frac{1}{n^{d+1}/(d+1)!}\sum\limits_{i=1}^{r_n}\widehat\mu_i(H^0_\R(nD),\xi_{ng})\\
&=(d+1)\lim\limits_{n\rightarrow +\infty}\frac{1}{ n^{d}/d!}\sum\limits_{i=1}^{r_n}\frac{\widehat\mu_i(H^0_\R(nD),\xi_{ng})}{n}\\
&=(d+1)\int_{\Delta(D)}G_{(D,g)}(x)dx.\end{aligned}
$$
\end{proof}

As an application of above result, we can reduce the study of $\vol_\chi(\cdot)$ to concave functions on convex bodies in $\R^d$.

\subsection{Variation of concave transforms}
In this subsection, we keep the same assumption and notations as in previous subsection. We are going to mainly study the variation of $\inf\limits_{x\in\Delta(D)^\circ}G_{(D,g)}(x)$ which would lead to the integrablity of the concave transform under certain circumstances.

\begin{lemm}\label{lem_bottom_concave}
Let $(D,g)$ be an adelic $\R$-Cartier divisor. Then for any $\zeta\in \R$ we have
$\zeta\leq\inf\limits_{x\in\Delta(D)}G_{(D,g)}(x)$ if and only if $\Delta(\Gamma_{(D,g)}^\zeta)=\Delta(D)$.
\end{lemm}
\begin{proof}
This comes from the construction of $G_{(D,g)}(x)$ i.e.
$$G_{(D,g)}(x)=\sup\{t\in\R| x\in \Delta(\Gamma_{(D,g)}^t)\}.$$
for any $x\in\Delta(D)$ and the fact that $\{\Gamma_{(D,g)}^t\}_{t\in\R}$ is a decreasing family of convex bodies, we obtain that $x\in\Gamma_{(D,g)}^t$.
\end{proof}

\begin{lemm}
Let $(D,g)$ be an adelic $\R$-Cartier divisor and $\phi$ be an integrable function on $\Omega$. Then $$G_{(D,g+\phi)}(x)=G_{(D,g)}+\int_{\omega\in\Omega}\phi(\omega)d\omega.$$
\end{lemm}
\begin{proof}
Set $A=\displaystyle\int_{\omega\in\Omega}\phi(\omega)d\omega$.
Since for each $n\in\N_+$, $$\{\widehat{\mu}_i(H^0_\R(nD),\xi_{ng+n\phi})\}_{1\leqslant i\leqslant r_n}=\{\widehat{\mu}_i(H^0_\R(nD),\xi_{ng})+nA\}_{1\leqslant i\leqslant r_n}$$
is identified with the sorted sequence $\{g_{(D,g+\phi)}(n,\alpha)\}_{\alpha\in\Gamma_{n}}$, it follows that 
$$g_{(D,g+\phi)}(x)=g_{(D,g)}(x)+nA.$$
By the construction decribed in previous subsection of concave transform, $\widetilde{g}_{(D,g+\phi)}(x)=\widetilde{g}_{(D,g)}(x)+nA$ which would lead to $G_{(D,g+\phi)}(x)=G_{(D,g)}(x)+A$ for each $x\in\Delta(D)$.
\end{proof}

\begin{prop}\label{prop_homogeneity}
Let $(D,g)$ be an adelic $\R$-Cartier divisor on $X$ and $\alpha\in\R_+$. The following homogeneity holds:
$$\alpha \inf\limits_{x\in\Delta(D)} G_{(D,g)}(x)=\inf\limits_{x\in\Delta(\alpha D)}G_{(\alpha D,\alpha g)}(x).$$
\end{prop}
\begin{proof}
Take an integrable function $\phi$ on $\Omega$ such that $\displaystyle\int_{\omega\in\Omega}\phi(\omega)d\omega=1.$
For any $t\in\R$, set
$$\begin{aligned}H_t(D,g):=\vol(D,g-t\phi)&=\int_{x\in \Delta(D)^\circ}\max\{G_{(D,g-t\phi)}(x),0\}dx\\
&=\int_{x\in\Delta(D)^\circ}\max\{G_{(D,g)}(x)-t,0\}dx.
\end{aligned}$$
It's easy to see that
\begin{equation}\label{eq_differential}
\frac{d}{dt}H_t(D,g)=-\eta(\Delta(\Gamma^t_{(D,g)}))
\end{equation}
where $\eta$ is the Lebesgue measure on $\R^d$.

Apply (\ref{eq_differential}) to $(\alpha D,\alpha g)$, we obtain
\begin{equation}\frac{d}{dt}H_{\alpha t}(\alpha D,\alpha g)=-\alpha \eta(\Delta(\Gamma^{\alpha t}_{(\alpha D, \alpha g)})).\end{equation}
On the other hand, since 
$$H_{\alpha t}(\alpha D,\alpha g)=\vol(\alpha D,\alpha g-\alpha t \phi)=\alpha^{d+1}\vol(D,g-t\phi)=\alpha^{d+1}H_t(D,g),$$
we can deduce that
\begin{equation}\frac{d}{dt}H_{\alpha t}(\alpha D,\alpha g)=-\alpha^{d+1}\eta(\Delta(\Gamma_{(D,g)}^t)).\end{equation}
Thus \begin{equation}\label{eq_homogeneity}\eta(\Delta(\Gamma^t_{(\alpha D,\alpha g)}))=\alpha^d\eta(\Delta(\Gamma^t_{(D,g)})).\end{equation}
Take a real number $t$ such that $t\leq \inf\limits_{x\in\Delta(D)} G_{(D,g)}(x)$. Then $\Delta(\Gamma_{(D,g)}^t)=\Delta(D)$ follows from Lemma \ref{lem_bottom_concave}.
Then by the fact that $$\eta(\Delta(\alpha D))=\mathrm{vol}(\alpha D)=\alpha^d\mathrm{vol}(D)=\alpha^d\eta(\Delta(D))$$ and (\ref{eq_homogeneity}), it holds that
$$\Delta(\Gamma^{\alpha t}_{(\alpha D,\alpha g)})=\Delta(\alpha D).$$
Hence $\alpha t\leq \inf\limits_{x\in\Delta(\alpha D)}G_{(\alpha D,\alpha g)}(x)$. As $t$ is arbitrary, in consequence, $$\alpha \inf\limits_{x\in\Delta(D)} G_{(D,g)}(x)\leqslant\inf\limits_{x\in\Delta(\alpha D)}G_{(\alpha D,\alpha g)}(x).$$

At last, we replace $\alpha$ by $\displaystyle\frac{1}{\alpha}$ and $(D,g)$ by $(\alpha D,\alpha g)$, we get the other direction of the above inequality.
\end{proof}

\begin{defi}
For any $\overline D\in\widehat{\Div}_\R(X)$ such that $D$ is big and $G_{\overline D}(x)$ is bounded from below, we define the following alternative of $\vol_\chi(\overline D)$:
$$\vol_I(\overline D):=\int_{x\in\Delta(D)}G_{\overline D}(x)dx.$$
\end{defi}
\begin{rema}
By Corollary \ref{cor_converge}, we know that if $\mumin^{\inf}(D,g)>-\infty$ (for example $D$ is ample, see lemma \ref{lem_bottom}), then 
\begin{equation}\label{eq_relationship_vols}\vol_\chi(D,g)=(d+1)\vol_I(D,g).\end{equation}
Moreover, if $\inf\limits_{x\in\Delta(D)}G_{\overline D}(x)\geqslant 0$, then
$$\vol(D,g)=(d+1)\vol_I(D,g).$$
Then we can actually see that (\ref{eq_relationship_vols}) also holds for $\Q$-ample $\Q$-Cartier divisors due to Proposition \ref{prop_vol_vol_chi}. In details, there exists a integrable function $\phi$ on $\Omega$ such that $\mumin^{\sup}(D,g+\phi)>0$ and $\inf\limits_{x\in \Delta(D)} G_{(D,g+\phi)}(x)>0$, then
$$\vol_\chi(D,g+\phi)=\vol(D,g+\phi)=(d+1)\vol_I(D,g+\phi).$$
By removing $\phi$ from above equations, we get (\ref{eq_relationship_vols}) due to the shifting property and Lemma \ref{lem_bottom_concave}.
\end{rema}

\begin{prop}
Let $\overline D=(D,g)$ be an adelic $\R$-Cartier divisor. If $D$ is big and $\inf\limits_{x\in\Delta(D)}G_{\overline D}(x)>-\infty$, then 
$$\vol_I(\alpha\overline D)=\alpha^{d+1} \vol_I(\overline D)$$
for any $\alpha\in\R_+$.
\end{prop}
\begin{proof}
Take an integrable function $\phi$ on $\Omega$ such that
$$\int_{\omega\in\Omega} \phi(\omega)d\omega >-\inf\limits_{x\in\Delta(D)}G_{\overline D}(x).$$
Then 
\begin{equation}\label{eq_homogeneity_1}\vol_I(\overline D)+\mathrm{vol}(D)\int_{\omega\in\Omega} \phi(\omega)d\omega=\vol_I(\overline D+(0,\phi))=\frac{1}{d+1}\vol(\overline D+(0,\phi)).\end{equation}
By Proposition \ref{prop_homogeneity}, we know that 
$$\int_{\omega\in\Omega} \alpha\phi(\omega)d\omega >-\inf\limits_{x\in\Delta(\alpha D)}G_{\alpha\overline D}(x).$$
Thus \begin{equation}\begin{aligned}\label{eq_homogeneity_2}\vol_I(\alpha \overline D)+&\mathrm{vol}(\alpha D)\int_{\omega\in\Omega} \alpha\phi(\omega)d\omega=\\&\vol_I(\alpha\overline D+(0,\alpha\phi))=\frac{1}{d+1}\vol(\alpha \overline D+(0,\alpha \phi)).\end{aligned}\end{equation}

By the fact that 
$\vol(\alpha \overline D)=\alpha^{d+1}\vol(\overline D)$ and $\mathrm{vol}(\alpha D)=\alpha^{d}\mathrm{vol}(D)$,
we deduce the result from (\ref{eq_homogeneity_1}) and (\ref{eq_homogeneity_2}).
\end{proof}

The following inequalities about $\vol_\chi(\cdot)$ mainly derived from Chen \cite{chen:hal}, especially the key lemma can be found in \cite[Theorem 2.3]{chen:hal}.
\begin{lemm}\label{lem_ineq_expect}
Let $\Delta_1$ and $\Delta_2$ be convex bodies in $\R^d$, $G_1$, $G_2$ and $G$ be upper bounded measurable functions on $\Delta_1$, $\Delta_2$ and $\Delta_1+\Delta_2$ respectively. Further, we assume that $G$ is positive and $G(x+y)\geqslant G_1(x)+G_2(y)$. Then it holds that
$$\frac{\displaystyle\int_{\Delta_1+\Delta_2}G(x)dx}{\eta(\Delta_1+\Delta_2)}\geqslant \frac{\displaystyle\int_{\Delta_1}G_1(x)dx}{\eta(\Delta_1)}+\frac{\displaystyle\int_{\Delta_2}G_2(x)dx}{\eta(\Delta_2)}$$
where $\eta$ is the Lebesgue measure on $\R^d$.
\end{lemm}
\begin{proof}
Let $Z_1$ be a random variable on the uniformly distributed probability space $(\Delta_1,\displaystyle\frac{\mu}{\mu(\Delta_1)})$ and $Z_2$ be a random variable on the uniformly distributed probablity space $(\Delta_2,\displaystyle\frac{\mu}{\mu(\Delta_2)})$. Then the inequality comes from the following inequality for expectations:
$$\mathbb{E}[G(Z_1+Z_2)]\geqslant \mathbb{E}[G_1(Z_1)]+\mathbb{E}[G_2(Z_2)].$$
\end{proof}
\begin{theo}\label{ineq_big}
Let $(D_1,g_1)$ and $(D_2,g_2)$ be two $\R$-Cartier adelic divisors on $X$. We assume that 
\begin{enumerate}
    \item[(1)] both $D_1$ and $D_2$ are big,
    \item[(2)] $\inf\limits_{x\in\Delta(D_1)}G_{(D_1,g_1)}(x)>-\infty$ and $\inf\limits_{x\in\Delta(D_2)}G_{(D_2,g_2)}(x)>-\infty$.
\end{enumerate}
Then the following inequality holds:
$$
\frac{\vol_I(D_1,g_1)}{\mathrm{vol}(D_1)}+\frac{\vol_I(D_2,g_2)}{\mathrm{vol}(D_2)}\leqslant
\frac{\displaystyle\int_{\Delta(D_1)+\Delta(D_2)}G_{(D_1+D_2,g_1+g_2)}(x)dx}{\mathrm{vol}(\Delta(D_1)+\Delta(D_2))}.
$$
In particular, if $\dim(X)=1$, then
$$\frac{\vol_I(D_1,g_1)}{\mathrm{vol}(D_1)}+\frac{\vol_I(D_2,g_2)}{\mathrm{vol}(D_2)}\leqslant
\frac{\vol_I(D_1+D_2,g_1+g_2)}{\mathrm{vol}(D_1+D_2)}.$$
\end{theo}
\begin{proof}
By \cite[Proposition 6.3.28]{adelic} , we have the following two facts that 
\begin{enumerate}
    \item[(1)] $\Delta(D_1)+\Delta(D_2)\subset\Delta(D_1+D_2)$,
    \item[(2)] $G_{(D_1,g_1)}(x)+G_{(D_2,g_2)}(y)\leqslant G_{(D_1+D_2,g_1+g_2)}(x+y)$ for any $x\in\Delta(D_1)$ and $y\in\Delta(D_2)$.
\end{enumerate}

Then it follows from Theorem \ref{thm_converge} and Lemma \ref{lem_ineq_expect} that
$$
\frac{\vol_I(D_1,g_1)}{\mathrm{vol}(D_1)}+\frac{\vol_I(D_2,g_2)}{\mathrm{vol}(D_2)}\leqslant
\frac{\displaystyle\int_{\Delta(D_1)+\Delta(D_2)}G_{(D_1+D_2,g_1+g_2)}(x)dx}{\mathrm{vol}(\Delta(D_1)+\Delta(D_2))}.
$$

To obtain the second inequality, it suffices to show that $$\Delta(D_1)+\Delta(D_2)=\Delta(D_1+D_2).$$
In the case where $\dim(X)=1$, the Okoukov bodies $\Delta(D_1)$, $\Delta(D_2)$ and $\Delta(D_1+D_2)$ are bounded intervals of $\R$.
Note that $\Delta(D_1)+\Delta(D_2)\subset\Delta(D_1+D_2)$ and $$\mathrm{vol}(D_1+D_2)=\mathrm{deg}(D_1+D_2)=\mathrm{deg}(D_1)+\mathrm{deg}(D_2)=\mathrm{vol}(D_1)+\mathrm{vol}(D_2),$$ it follows that $\Delta(D_1)+\Delta(D_2)=\Delta(D_1+D_2)$. 
\end{proof}

\subsection{Applications on arithmetic surfaces}Now we focus on the case of arithmetic surfaces over adelic curves, i.e. $\dim X=1$. There is one thing worth noting that for any $\R$-Cartier divisor $D$ on $X$, if $\mathrm{deg}(D)>0$, then we can write $D$ in the form of 
$$D=\lambda_1 D_1+\lambda_2 D_2+\cdots+\lambda_n D_n$$
where $D_i$ are ample Cartier divisors and $\lambda_i$ are positive real numbers. This is just due to Nakai criterion for $\R$-divisors (see Lazarsfeld \cite[Theorem 2.3.18]{Positivity}). Or you can just see Theorem \ref{thm_nakai_curves} for specifically the curve case. By the fact that any Cartier divisor admits a Green function on it, we can write $(D,g)$ in the form of
$$(D,g)=\lambda_1(D_1,g_1)+\lambda_2(D_2,g_2)+\cdots+\lambda_n (D_n,g_n).$$

\begin{prop}\label{prop_bottom_for_G}
Assume that $\dim X=1$. Let $(D,g)$ be an adelic $\R$-Cartier divisor. If $\mathrm{deg}(D)>0$ then $$ \inf\limits_{x\in\Delta(D)} G_{(D,g)}(x)>-\infty.$$
\end{prop}
\begin{proof}
As discussed above, we can write $(D,g)$ in the form of
$$(D,g)=\lambda_1(D_1,g_1)+\lambda_2(D_2,g_2)+\cdots+\lambda_n (D_n,g_n).$$
Since $D_i$ is ample Cartier for $i=1,\dots,n$, by Lemma \ref{lem_bottom}, we have $\mumin^{\inf}(D_i,g_i)>-\infty$ which implies that $$\inf\limits_{x\in\Delta(D)} G_{(D_i,g_i)}(x)>-\infty$$ 
for each $i=1,\dots,n$.
Moreover, by the homogeneity described in Proposition \ref{prop_homogeneity}, we get that
$$\inf\limits_{x\in\Delta(\lambda_i D)} G_{(\lambda_i D_i,\lambda_i g_i)}(x)>-\infty.$$

Since we have the two following facts that
\begin{enumerate}
    \item[(1)] $\Delta(\lambda_1 D_1)+\Delta(\lambda_2 D_2)=\Delta(\lambda_1 D_1+\lambda_2 D_2)$,
    \item[(2)] For any $x\in \Delta(\lambda_1 D_1)$ and $y\in\Delta(\lambda_2 D_2)$, it holds that $$G_{(\lambda_1 D_1,\lambda_1 g_1)}(x)+G_{(\lambda_2 D_2,\lambda_2 g_2)}(y)\leqslant G_{(\lambda_1 D_1+\lambda_2 D_2,\lambda_1 g_1+\lambda_2 g_2)}(x+y).$$
\end{enumerate}
One obtains that $$\begin{aligned}
\inf\limits_{x\in\Delta(\lambda_1 D_1+\lambda_2 D_2)}G_{(\lambda_1 D_1+\lambda_2 D_2,\lambda_1 g_1+\lambda_2 g_2)}&(x)\geqslant\\
\inf\limits_{x\in\Delta(\lambda_1 D_1)}G_{(\lambda_1 D_1,\lambda_1 g_1)}(x)&+\inf\limits_{x\in\Delta(\lambda_2 D_2)}G_{(\lambda_2 D_2,\lambda_2 g_2)}(x)>-\infty.
\end{aligned}$$
Then proceeding by induction on $n$, we obtain the assertion.
\end{proof}

Now we know that for any ample $\R$-Cartier divisor $D$ on arithmetic surface $X$, $\vol_I(D,g)$ is well defined. In the following, we are going to prove the continuity of $\vol_I(\cdot)$.

\begin{prop}\label{mod_big2} Assume that $\dim X=1$.
Let $(D,g)$ be an $\R$-Cartier adelic divisor such that $\vol_I(D,g)> 0$ and $(E,h)$ be an $\R$-Cartier adelic divisor. Then there exists a positive integer $n_0$ such that $\vol_I(n(D,g)+(E,h))> 0$ for $n\geqslant n_0$.
\end{prop}
\begin{proof}
Take an integer $n_0$ such that $\mathrm{deg}(n_0D+E)>0$.
Then there exists a non-negative integrable function $\phi$ on $\Omega$ such that $\vol_I(n_0D+E,n_0g+h+\phi)\geqslant 0$. 

Since $\vol_I(D,g)>0$, there exist an integer $m_0$ such that for $m\geqslant m_0\in \N$, it holds that $\vol_I(D,g-\frac{\phi}{m})>0$.
Hence $\vol_I(mD,mg-\phi)=m^{d+1}\vol_I(D,g-\frac{\phi}{m})>0$.

We can therefore get that $$\frac{\vol_I((m+n_0)D+E,(m+n_0)g+h)}{\mathrm{vol}((m+n_0)D+E)}> \frac {\vol_I(n_0D+E,n_0g+h+\phi)}{\mathrm{vol}(n_0D+E)}\geqslant 0$$ by Theorem \ref{ineq_big}.
Thus $\vol_I(nD+E,ng+h)\geqslant 0$ for any $n\geqslant m_0+n_0$.
\end{proof}

\begin{theo}\label{thm_continuity_vol_I}Assume that $\dim X=1$. 
For any $\R-$Cartier divisors $\overline D=(D,g),\overline E_1,\dots, \overline E_n$, if $\mathrm{deg}(D)> 0$, then$$\lim\limits_{\lvert\epsilon_1\rvert+\cdots+\lvert\epsilon_n\rvert \rightarrow 0}\vol_I(\overline D+ \epsilon_1\overline E_1+\cdots+\epsilon_n\overline E_n)=\vol_I(\overline D).$$
\end{theo}
\begin{proof}
By Lemma \ref{lem_multi_phi}, there exists a Green function $g'$ on $D$ such that $\vol_I(D,g')>0$.
Then there exists a sufficiently large $a\in \N$ such that $\vol_I(a(D,g')\pm \overline E_i)>0$ for any $i=1,\dots,n$ due to Proposition \ref{mod_big2}.

If we set $\epsilon=\lvert \epsilon_1\rvert+\cdots+\lvert \epsilon_n\rvert$, then for $\epsilon\ll 1$, the inequality $$\begin{aligned}&\frac{\vol_I(D+a\epsilon D,g+a\epsilon g')}{\mathrm{vol}(D+a\epsilon D)}\geqslant \\&\frac{\vol_I(\overline D+ \epsilon_1\overline E_1+\cdots+\epsilon_n\overline E_n)}{\mathrm{vol}(D+ \epsilon_1E_1+\cdots+\epsilon_n E_n)} \geqslant
\frac{\vol_I(D-a\epsilon D,g-a\epsilon g')}{\mathrm{vol}(D-a\epsilon D)}\end{aligned}$$
holds by Theorem \ref{ineq_big}.

Since $$\begin{aligned}\vol_I(D+a\epsilon D,g+a\epsilon g')&=
\vol_I((1+a\epsilon)D,(1+a\epsilon)g+a\epsilon(g'-g))\\
&=(1+a\epsilon)^{d+1}\vol_I(D,g+\frac{a\epsilon}{1+a\epsilon}(g'-g))
\end{aligned}$$ for $\epsilon$ small enough,
we obtain that $\lim\limits_{\epsilon\rightarrow 0}\vol_I(D+a\epsilon D,g+a\epsilon g')=\vol_I(D,g)$ by Lemma \ref{mltply_by_fun}.

We can run the similar process to get that $\lim\limits_{\epsilon\rightarrow 0}\vol_I(D-a\epsilon D,g-a\epsilon g')=\vol_I(D,g)$. The continuity of $\vol_I(\cdot)$ is thus proved.
\end{proof}

\begin{rema}
Someone may wonder whether we can deduce the continuity of $\vol_\chi(\cdot)$ from the continuity of $\vol_I(\cdot)$ in the case where $X$ is of dimension $1$. Obviously, if $\mumin^{\sup}(D,g)>-\infty$ for any $(D,g)\in \Div_\R(X)$ with $\mathrm{deg}(D)>0$, then the answer is yes.

But even we already have the result that $\inf\limits_{x\in\Delta(D)}G_{(D,g)}(x)>-\infty$, we can not show that $\vol_I(D,g)=2\vol_\chi(D,g)$ from the weak convergence in Theorem \ref{thm_concave_trans}.

Assume that $\{r_n\}_{n\in\N}$ is a strictly increasing sequence of postivie integers. 
For any $x$, we denote by $\delta_{x}$ the measure with mass $1$ on the point $X$.
Then the sequence of empirical measures $\{\eta_n:=\displaystyle\frac{1}{r_n}\delta_{-r_n}+\displaystyle\frac{r_n-1}{r_n}\delta_1\}$ weakly converges to $\eta=\delta_{1}$ which is of compact support.
This is simply due to the difference between weak convergence and convergence. Take $f(x)=x$, then $\limnto \displaystyle\int_\R f(x)\eta_n(dx)=\limnto{\big(\displaystyle\frac{-r_n}{r_n}+\displaystyle\frac{r_n-1}{r_n}\big)}=0$, but 
$\displaystyle\int_\R f(x)\eta(dx)=1$.
\end{rema}

Instead of the continuity of $\vol_\chi(\cdot)$, we can view $\vol_I(\cdot)$ as a continous extension of $\vol_\chi(\cdot)$ by the following corollary.
\begin{coro}\label{cor_conti_ext_chi}
Assume that $\dim X=1$. Let $\overline D=(D,g)$ be an adelic $\R$-Cartier divisor on $X$ such that $\mathrm{deg}(D)>0$. So we can write $\overline D$ as $$\overline D=\alpha_1\overline D_1+\alpha_2 \overline D_2+\cdots+\alpha_n \overline D_n$$
where $\overline D_i=(D_i,g_i)$, $D_i$ is an ample Cartier divisor with $\alpha_i>0$, $i=1,\dots, n.$
Then we can view $\vol_I(\cdot)$ as a continuous extension of $\vol_\chi(\cdot)$ according to the following formula:
$$\lim\limits_{\substack{a_i\rightarrow \alpha_i\\a_i\in\Q\\ i=1,\dots,n}}\vol_\chi(a_1 \overline D_1+a_2\overline D_2+\cdots+a_n\overline D_n)=\frac{\vol_I(\overline D)}{2}.$$
\end{coro}
\begin{proof}
This is due to the fact that $\vol_\chi(\cdot)$ agrees with $\displaystyle\frac{1}{2}\vol_I(\cdot)$ for any adelic $\Q$-Cartier $\Q$-ample divisor and the continuity of $\vol_I(\cdot)$.
\end{proof}

\section{Appendix}
This part is dedicated to give proofs for some results in algebraic geometry, especially the surjectivity of multiplication maps for ample divisors which is just taken from \cite{Positivity}. Then we apply this property to Okounkov bodies.

\begin{lemm}\label{lem_resolution}
Let $X$ be a projective scheme and $D$ an ample Cartier divisor on $X$. Then any coherent sheaf $\mathcal F$ on $X$ admits a (possibly non-terminating) resolution:
$$\dots\rightarrow \oplus \mathcal O_X(-p_1D)\rightarrow \oplus \mathcal O_X(-p_0D)\rightarrow \mathcal F\rightarrow 0$$
where $0<p_0<p_1<\cdots$.
\end{lemm}
\begin{proof}
Since $D$ is ample, we have $\mathcal F\otimes \mathcal O_X(p_0D)$ is globally generated for some $p_0\gg0$.
Then we can get a surjective map $\oplus \mathcal O_X\rightarrow \mathcal F\otimes \mathcal O_X(p_0D)$ which induces the surjective map $\oplus \mathcal O_X(-p_0D)\rightarrow \mathcal F$. For the kernal of the map, we can apply the same process the get $p_1$, and then continue.
\end{proof}

\begin{lemm}
Let $X$ be a projective scheme. Consider a resolution of coherent sheaves:
$$0\rightarrow \mathcal F_n\rightarrow \cdots \rightarrow \mathcal F_1\rightarrow \mathcal F_0\rightarrow \mathcal F\rightarrow 0.$$
If 
$$ H^k(X,\mathcal F_0)=H^{k+1}(X,\mathcal F_1)=\cdots=H^{k+n}(X,\mathcal F_n)=0,
$$
then $H^k(X,\mathcal F)=0$.
\end{lemm}

\begin{theo}\label{thm_surjectivity}
Let $X$ be a projective scheme, and let $D$ and $E$ be ample Cartier divisors and $B$ a Cartier divisor on $X$. Then there exists an $N\in \N_+$ which is related to $D$ and $E$, such that for every $n,m\geqslant N$, the multiplication map
$$H^0(nD)\otimes H^0(mE+B)\rightarrow H^0(nD+mE+B)$$ is surjective.
\end{theo}
\begin{proof}
Let $\Delta\subset X\times X$ denote the image of $X$ under the diagonal morphism.
Consider the exact sequence of sheaves on $X\times X$:
$$0\longrightarrow \mathcal I_\Delta\longrightarrow \mathcal O_{X\times X}\longrightarrow \mathcal O_\Delta\longrightarrow 0.$$

Set $(nD,mE+B)=p_1^*(nD)+p_2^*(mE+B)$. Denote by $\mathcal O_{X\times X}(nD,mE+B)$ the structure sheaf of $(nD,mE+B)$, and set 
$$\begin{aligned}\mathcal I_\Delta(nD,mE+B)&:=\mathcal I_\Delta\otimes \mathcal O_{X\times X}(nD,mE+B),\\
\mathcal{O}_\Delta(nD,mE+B)&:=\mathcal O_\Delta\otimes \mathcal O_{X\times X}(nD,mE+B).
\end{aligned}
$$
Then there exists an exact sequence of sheaves on $X\times X$ that
$$0\longrightarrow \mathcal I_\Delta(nD,mE+B)\longrightarrow \mathcal O_{X\times X}(nD,mE+B)\longrightarrow \mathcal O_\Delta(nD,mE+B)\longrightarrow 0$$
which leads to the exact sequence
$$\begin{aligned}0\rightarrow &H^0(X\times X,\mathcal I_\Delta(nD,mE+B))\rightarrow H^0(X\times X,\mathcal O_{X\times X}(nD,mE+B))\rightarrow\\ &H^0(X\times X,\mathcal O_\Delta(nD,mE+B))
\rightarrow H^1(X\times X,\mathcal I_\Delta(nD,mE+B))\rightarrow ...\text{ }.
\end{aligned}$$

Since $H^0(X\times X,\mathcal O_{X\times X}(nD,mE+B))=H^0(X,nD)\otimes H^0(X,mE+B)$ by K\"unneth formula and $H^0(X\times X,\mathcal O_\Delta(nD,mE+B))=H^0(X,nD+mE+B)$, it suffices to show that $H^1(X\times X,\mathcal I_\Delta(nD,mE+B))=0$ for every $n,m\gg 0$.

By Lemma \ref{lem_resolution}, we can construct a resolution with the form
$$\cdots\rightarrow \oplus \mathcal{O}_{X\times X}(-p_1 D,-p_2 E)\rightarrow \oplus\mathcal O_{X\times X}(-p_0 D,-p_0 E)\rightarrow \mathcal I_\Delta\rightarrow 0$$
where $0<p_0<p_1<\cdots$.

Thus it suffices to show that 
$$H^{i+1}\left(X\times X, \mathcal O_{X\times X}\big((n-p_{i})D,(m-p_{i})E+B\big)\right)=0$$
for any $i\geqslant 0$ and $n,m\gg 0$.
For $i\geqslant \dim X\times X$, this is trivial due to Grothendieck's vanishing theorem. In the case that $i< \dim X\times X$, by K\"unneth formula, we have
$$\begin{aligned}H^{i+1}\left(X\times X, \mathcal O_{X\times X}\big((n-p_{i})D,(m-p_{i})E+B\big)\right)&=\\
\mathop\oplus\limits_{j+k=i+1}H^j(X,(n-p_i)D)\otimes &H^k(X,(m-p_i)E+B).\end{aligned}$$

As an application of Serre's vanishing theorem in \cite[Proposition 5.3]{Hartshorne}, we know that there exists an $N\in \N_+$ such that for $n>N$,
$H^l(X,(n-p_i)D)=0$ for any $l>0$ and $i<\dim X\times X$.

Therefore we deduce the surjectivity of multiplication map.
\end{proof}

\begin{theo}\label{thm_nakai_curves}
Let $X$ be an integral projective curve, and $D$ an $\R$-divisor with positive degree. Then we can write $D$ in the form of
$$D=a_1 D_1+\cdots+a_m D_m$$
where $D_i$ are ample divisors and $a_i>0$.
\end{theo}
\begin{proof}
Assume that $D=n_1 P_1+\cdots+n_r P_r$ where $P_i$ are closed points on $X$.
Then we proceed by induction on $r$. We may further assume that $$n_1\geqslant n_2\geqslant\cdots\geqslant n_r.$$
In the case that $r=1$, it's trivial. In the following we assume that $r>1$.
If $n_r\geqslant 0$, then we are done because $D$ is already in desired form. If $n_r< 0$, set
$D'=n_1 P_1+\cdots+n_{r-1} P_{r-1}$. Then $\mathrm{deg}(D')>\mathrm{deg}(D)>0$, by the induction hypothesis, we can write $D'$ in the form that
$$D'=a_1 D_1+\cdots+a_m D_m$$ where $D_i$ are ample and $a_i >0$.

Then take rational numbers $\lambda_i<a_i/\lvert n_r\rvert$ for $i=1,\cdots, m$ such that 
$$\sum\limits_{i=1}^m\lambda_i\mathrm{deg}(D_i)>\mathrm{deg}(P_r).$$
This can be acheived because $\mathrm{deg}(D')=\sum\limits_{i=1}^m a_i\mathrm{deg}(D_i)> -n_r\mathrm{deg}(P_r).$ Then we can write $D$ in the form that
$$D=\sum\limits_{i=1}^m(a_i-\lvert n_r\rvert \lambda_i)D_i+\lvert n_r\rvert\left(\sum\limits_{i=1}^m\lambda_i D_i-P_r\right).$$
\end{proof}
\bibliography{mybibliography}
\bibliographystyle{smfplain}
\end{document}